\documentclass[12pt,twoside,reqno]{amsart}\usepackage[cp1251]{inputenc}
\usepackage{amssymb}

\advance\textwidth45mm \advance\hoffset-25mm
\advance\textheight40mm \advance\voffset-20mm
\allowdisplaybreaks[2]
\theoremstyle{plain}
\newtheorem{theorem}{Theorem}[section]
\newtheorem{lemma}[theorem]{Lemma}
\newtheorem{corollary}[theorem]{Corollary}
\newtheorem{proposition}[theorem]{Proposition}
\theoremstyle{remark}
\newtheorem{remark}[theorem]{Remark}
\newtheorem{example}[theorem]{Example}

\newcommand\sign{\operatorname{sign}}

\begin{document}

\title[On the kernel of the $(\kappa,a)$-generalized Fourier
transform]{On the kernel of the $(\kappa,a)$-generalized Fourier
transform}

\author{D.~V.~Gorbachev}
\address{D.~V.~Gorbachev, Tula State University,
Department of Applied Mathematics and Computer Science,
300012 Tula, Russia}
\email{dvgmail@mail.ru}

\author{V.~I.~Ivanov}
\address{V.~I.~Ivanov, Tula State University,
Department of Applied Mathematics and Computer Science,
300012 Tula, Russia}
\email{ivaleryi@mail.ru}

\author{S.~Yu.~Tikhonov}
\address{S.~Yu.~Tikhonov, Centre de Recerca Matem\`atica\\
Campus de Bellaterra, Edifici~C 08193 Bellaterra (Barcelona), Spain; ICREA, Pg.
Llu\'is Companys 23, 08010 Barcelona, Spain, and Universitat Aut\`onoma de
Barcelona.}
\email{stikhonov@crm.cat}

\date{\today}

\thanks{The research of D.~Gorbachev and V.~Ivanov was performed by a grant of
RScF (project 18-11-00199),
https://rscf.ru/project/18-11-00199.
S.~Tikhonov was partially supported by
PID2020-114948GB-I00,  2017 SGR 358,
 the CERCA Programme of the Generalitat de Catalunya,
 Severo Ochoa and Mar\'{i}a de Maeztu Program for Centers and Units of Excellence in R$\&$D (CEX2020-001084-M),
and  Ministry of Education and Science of the Republic of Kazakhstan (AP08856479).}

\begin{abstract}
For  the kernel  $B_{\kappa,a}(x,y)$ of the $(\kappa,a)$-generalized Fourier
transform $\mathcal{F}_{\kappa,a}$, acting in $L^{2}(\mathbb{R}^{d})$ with the
weight $|x|^{a-2}v_{\kappa}(x)$, where $v_{\kappa}$~is the Dunkl weight, we
study the important question of when
$\|B_{\kappa,a}\|_{\infty}=B_{\kappa,a}(0,0)=1$. The positive answer was known
for $d\ge 2$ and $\frac{2}{a}\in\mathbb{N}$. We investigate the case $d=1$ and
$\frac{2}{a}\in\mathbb{N}$. Moreover, we give sufficient conditions on
parameters for $\|B_{\kappa,a}\|_{\infty}>1$ to hold with $d\ge 1$ and any $a$.

We also study the image of the Schwartz space under the
$\mathcal{F}_{\kappa,a}$ transform. In particular, we obtain that
$\mathcal{F}_{\kappa,a}(\mathcal{S}(\mathbb{R}^d))=\mathcal{S}(\mathbb{R}^d)$
only if $a=2$. Finally, extending the Dunkl transform, we introduce
non-deformed transforms generated by $\mathcal{F}_{\kappa,a}$ and study their
main properties.
\end{abstract}

\maketitle
\section{Introduction}
Let as usual $\Delta$ be the Laplacian operator  in $\mathbb{R}^{d}$.
For the Fourier transform 
\[
\mathcal{F}(f)(y)=(2\pi)^{-d/2}\int_{\mathbb{R}^{d}}f(x)e^{-i\langle x,y\rangle}\,dx
\]
Howe \cite{Ho87} obtained the following spectral decomposition of  $\mathcal{F}$
using the harmonic oscillator $-({\Delta-|x|^2})/2$
and its eigenfunctions forming the basis in $L^{2}(\mathbb{R}^{d})$:
\[
\mathcal{F}=\exp\Bigl(\frac{i\pi
d}{4}\Bigr)\exp\Bigl(\frac{i\pi}{4}\,\bigl(\Delta-|x|^{2}\bigr)\Bigr).
\]
Among other applications,
this decomposition is useful 
  to define the fractional power of Fourier transform; see \cite{SKO12,KM11}.

During last 30 years, a lot of attention has been given to various generalizations of the Fourier transform.
As an important example, to develop harmonic analysis on weighted spaces, the  Dunkl transform was introduced in \cite{Du92}.
  The Dunkl transform  $\mathcal{F}_\kappa$ is defined  with the help of a root system $\Omega\subset\mathbb{R}^{d}$, a reflection group $G\subset O(d)$,
and multiplicity function $\kappa\colon \Omega\to \mathbb{R}_+$ such that $\kappa$ is $G$-invariant.
 Here $G$ is
 generated by reflections $\{\sigma_{\alpha}\colon \alpha\in \Omega\}$,
where $\sigma_{\alpha}$ is a reflection with respect to hyperplane $(\alpha,x)=0$.

The differential-difference Dunkl Laplacian operator $\Delta_{\kappa}$ plays
the role of the classical Laplacian \cite{Ro02}. If $\kappa\equiv 0$, we have
$\Delta_{\kappa}=\Delta$. Dunkl Laplacian allows us to define the Dunkl
harmonic oscillator $\Delta_{\kappa}-|x|^{2}$ and the Dunkl transform
\[
\mathcal{F}_{\kappa}=\exp\Bigl(\frac{i\pi}{2}\,\Bigl(\frac{d}{2}+\langle \kappa\rangle\Bigr)\Bigr)
\exp\Bigl(\frac{i\pi}{4}\bigl(\Delta_{\kappa}-|x|^{2}\bigr)\Bigr),
\]
where $\langle \kappa\rangle=\frac{1}{2}\sum_{\alpha\in
\Omega}\kappa(\alpha)$.

Further extensions of Fourier and Dunkl transforms were obtained by Ben Sa\"{\i}d, Kobayashi, and {\O}rsted
in  \cite{SKO12}.
They defined the $a$-deformed Dunkl harmonic oscillator
\[
\Delta_{\kappa,a}=|x|^{2-a}\Delta_{\kappa}-|x|^{a},\quad a>0,
\]
and the  $(\kappa,a)$-generalized Fourier transform
\begin{equation}\label{eq1}
\mathcal{F}_{\kappa,a}=\exp\Bigl(\frac{i\pi}{2}\,(\lambda_{\kappa,a}+1)\Bigr)
\exp\Bigl(\frac{i\pi}{2a}\,\Delta_{\kappa,a}\Bigr),
\end{equation}
which is
a two-parameter family of unitary operators in $L^{2}(\mathbb{R}^{d},d\mu_{\kappa,a})$ equipped with the norm
\[
\|f\|_{2,d\mu_{\kappa,a}}=\Bigl(\int_{\mathbb{R}^{d}}|f(x)|^{2}\,d\mu_{\kappa,a}(x)\Bigr)^{1/2}.
\]
Here
\[
\lambda_{\kappa, a}=\frac{2\lambda_{\kappa}}{a},\quad \lambda_{\kappa}=\langle \kappa\rangle+\frac{d-2}{2},\quad
d\mu_{\kappa,a}(x)=c_{\kappa,a}v_{\kappa,a}(x)\,dx,\quad
v_{\kappa,a}(x)=|x|^{a-2}v_{\kappa}(x),
\]
\[
v_{\kappa}(x)=\prod_{\alpha\in \Omega}|\langle\alpha,x\rangle|^{\kappa(\alpha)},\quad
c^{-1}_{\kappa,a}=\int_{\mathbb{R}^{d}}e^{-|x|^{a}/a}v_{\kappa,a}(x)\,dx.
\]
Throughout the paper, we assume that   $
d+2\langle \kappa\rangle+a-2=2\lambda_{\kappa}+a>0$
or, equivalently, $\lambda_{\kappa, a}>-1$. Note that under this condition the weight function $v_{\kappa,a}$ is locally integrable.

For  $a=2$,  \eqref{eq1} reduces to  the Dunkl transform, while if  $a=2$ and $\kappa\equiv 0$, then  \eqref{eq1} is the classical Fourier transform.
 For $a\ne 2$, we arrive at deformed Dunkl and 
  Fourier transforms, which have
 various applications. 
  In particular, for  $a=1$ and $\kappa\equiv 0$
 deformed Dunkl transform
is the unitary inversion operator of the Schr\"{o}dinger model of minimal
representation of the group $O(N+1,2)$ \cite{KM11}.

The  unitary operator  $\mathcal{F}_{\kappa,a}$ on
$L^{2}(\mathbb{R}^{d},d\mu_{\kappa,a})$ can be written as the integral
transform  \cite[(5.8)]{SKO12}
\begin{equation*}
\mathcal{F}_{\kappa,a}(f)(y)=\int_{\mathbb{R}^{d}}B_{\kappa,a}(x,y)f(x)\,d\mu_{\kappa,a}(x)
\end{equation*}
with the continuous symmetric kernel  $B_{\kappa,a}(x,y)$ satisfying $B_{\kappa,a}(0,y)=1$. In particular,  $B_{0,2}(x,y)=e^{-i\langle x,y\rangle}$.

One of the fundamental questions in the theory of deformed transforms is to
investigate basic  properties of the kernel $B_{\kappa,a}(x,y)$, in
particular, to know when it is uniformly bounded. To illustrate the importance
of this property, note that  the condition $|B_{\kappa,a}(x,y)|\le M$ implies the
Hausdorff-Young inequality
\[
\bigl\|\mathcal{F}_{\kappa,a}(f)\bigr\|_{p',d\mu_{k,a}}\le
M^{2/p-1}\bigl\|f\bigr\|_{p,d\mu_{\kappa,a}},\quad
1\le p\le 2,\quad \frac{1}{p}+\frac{1}{p'}=1.
\]
A more important problem is to describe parameters so that there holds
\begin{equation}\label{eq2}
\|B_{\kappa,a}\|_{\infty}=\sup_{x,y\in\mathbb{R}^{d}}|B_{\kappa,a}(x,y)|=B_{\kappa,a}(0,0)=1.
\end{equation}
In this  case the Hausdorff-Young inequality  holds with the constant $1$ and one can define the generalized translation operator $\tau^yf(x)$ in
$L^{2}(\mathbb{R}^{d},d\mu_{k,a})$  by 
\[
\mathcal{F}_{\kappa,a}(\tau^yf)(z)=B_{\kappa,a}(y,z)\mathcal{F}_{\kappa,a}(f)(z)
\]
(see \cite{GIT16}), and moreover, its norm equals  1.

Let us list the known cases when \eqref{eq2} holds:

\smallbreak
\textbullet\ \ for $a=2$ \cite{Ro99};

\smallbreak
\textbullet\ \ for $a=1$ and either $d=1$, $\langle\kappa\rangle\ge 1/2$ or $d\ge 2$, $\langle\kappa\rangle\ge 0$ \cite[Propositions~5.10, 5.11]{SKO12}, \cite[Sect. 6]{GIT16};

\smallbreak
\textbullet\ \ for $\frac{2}{a}\in\mathbb{N}$ and $d\ge 2$, $ \langle\kappa\rangle\ge 0$ \cite{BIE13,CBL18}.\;\footnote{The case $ \langle\kappa\rangle> 0$ was announced in \cite[Remark 3]{CBL18}. The proof is similar to the one of \cite[Theorem 9]{CBL18}.}

\smallbreak
In this paper we continue to study the case
\[
\frac{2}{a}\in\mathbb{N}.
\]
Its importance was discussed in    \cite{SKO12}, \cite{BNS}, and \cite{CBL18}.
For $a=2$, \
$\mathcal{F}_{\kappa,a}$
reduces to the Dunkl transform and \eqref{eq2} is valid.

Our first goal in this paper is, on the one hand,  to extend the list  of parameters for which
\eqref{eq2} holds for $\frac{2}{a}\in\mathbb{N}$; on the other hand, 
 to point out  the cases when \eqref{eq2} does not hold.
The following theorem  describes  positive results, 
 where, for completeness, we include all known cases.

\begin{theorem}[{see \cite[$d\ge 2$]{CBL18}}]\label{thm2}
Let $0<a\le 1$, $\frac{2}{a}\in\mathbb{N}$. If $d=1$, $\langle \kappa\rangle\ge\frac{1}{2}$ or $d\ge 2$, $\langle \kappa\rangle\ge 0$, then equality \eqref{eq2} is true.
\end{theorem}

Our proof of Theorem  \ref{thm2} for $d=1$ is based on an integral representation of $B_{\kappa,a}$ with the special kernel and a study of positiveness of this kernel.
 This approach is closely related to the theory of positive definite functions.

In the general case  $d\ge 1$, we  give the  proof based on the approach developed
in  the papers  \cite{BL20,CBL18,DD21}, see Subsection~\ref{subsec4-1}, and the alternative proof based on representation with positive kernels, see Subsection~\ref{subsec4-2}.


With regard to negative results, we obtain the following theorem, where we specify 
  parameters when Theorem \ref{thm2} does not hold. 

\begin{theorem}\label{thm4}
In either of the following cases:

\smallbreak
\textbullet\ \
$d=1$, $0<a\le 1$, and $\langle\kappa\rangle=\frac{1}{2}-\frac{a}{4}$, or

\smallbreak
\textbullet\ \ $d\ge 1$, $a\in (1,2)\cup (2,\infty)$ and $\langle\kappa\rangle\ge 0$,
\\ we have 
\begin{equation}\label{eq3}
\|B_{\kappa,a}\|_{\infty}>1.
\end{equation}
\end{theorem}

The rest of the paper is organized as follows. Section 2 is devoted to the proof of
Theorem~\ref{thm2} in the case $d=1$.
In  Section~\ref{sec3}, we study the properties of the one-dimensional kernel $B_{\kappa,a}$ for
$\lambda_{\kappa,a}<0$. In particular, in
Subsection~\ref{subsec3-1} we investigate positive definiteness of kernels of the integral transforms generated by $\mathcal{F}_{\kappa,a}$.
In Section 4, we prove Theorem \ref{thm2} in full generality as well as  Theorem \ref{thm4} (Subsection~\ref{subsec4-2}).

In  Section~\ref{sec5} we study the question of how the $\mathcal{F}_{\kappa,a}$ transform  acts on  Schwartz functions.
The Schwartz space $\mathcal{S}(\mathbb{R}^{d})$ is invariant under the classical  Fourier transform $\mathcal{F}_{0,2}$ and
the Dunkl $\mathcal{F}_{\kappa,2}$ (see \cite{jeu}) but the case of deformed transforms is more complicated.
In fact we show that  
$\mathcal{S}(\mathbb{R}^d)$  is {\it not invariant} under $\mathcal{F}_{\kappa,a}$ for  $a\neq 2$, which contradicts a widely used statement in  \cite{Jo16}
(see Remark \ref{rem2---}).
If $\frac{a}{2}\notin\mathbb{N}$, then the generalized Fourier transform may  not be
infinitely differentiable, and if  $\frac{2}{a}\notin\mathbb{N}$, then
it may  not be rapidly decreasing at infinity. For $d=1$ and $\frac{2}{a}\in\mathbb{N}$,
the generalized Fourier transform of $f\in\mathcal{S}(\mathbb{R})$ is rapidly decreasing due to the representation
$\mathcal{F}_{\kappa,a}(f)(y)=F_1\bigl(|y|^{a/2}\bigr)+yF_2\bigl(|y|^{a/2}\bigr)$,
where the even functions  $F_1,F_2\in\mathcal{S}(\mathbb{R})$ (see Proposition~\ref{lem9}).

Finally, in Section \ref{sec6} we study one-dimensional non-deformed unitary transforms generated by~$\mathcal{F}_{\kappa,a}$:
\begin{equation*}
\mathcal{F}_{r}^{\lambda}(g)(v)=\int_{-\infty}^{\infty}e_{2r+1}(uv,\lambda)g(u)\,\frac{|u|^{2\lambda+1}
\,du}{2^{\lambda+1}\Gamma(\lambda+1)},
\end{equation*}
where $r\in\mathbb{Z}_+$, $\lambda\ge-1/2$, and the kernel
\[
e_{2r+1}(uv,\lambda)=j_{\lambda}(uv)+i(-1)^{r+1}\,\frac{(uv)^{2r+1}}{2^{2r+1}(\lambda+1)_{2r+1}}\,
j_{\lambda+2r+1}(uv)
\]
is an eigenfunction of the differential-difference operator
\[
\delta_{\lambda}g(u)=\Delta_{\lambda+1/2}g(u)-2r(\lambda+r+1)\,\frac{g(u)-g(-u)}{u^2}.
\]
Here $\Delta_{\lambda+1/2}$ is the one-dimensional Dunkl Laplacian for
$\langle \kappa\rangle=\lambda+\frac{1}{2}$.
Note that such unitary transforms give new examples of an important class of Bessel-Hankel type   transforms  with the kernel  $k(uv)$, see, e.g.,  \cite[Chap.~VIII]{Ti48}. In particular, they generalize the one-dimensional Dunkl transform ($r=0$).

\vspace{0.6mm}
\section{Proof of Theorem~\ref{thm2} in the one-dimensional case}\label{sec2}

In what follows, we assume that
\[
d=1,\quad a>0,\quad \kappa=\langle\kappa\rangle\ge 0,\quad \lambda_{\kappa}=\kappa-1/2,\quad 2\lambda_{\kappa}+a>0,\quad\lambda=\lambda_{\kappa,a}=2\lambda_{\kappa}/a,
\]
\[
v_{\kappa,a}(x)=|x|^{2\kappa+a-2},\quad
d\mu_{\kappa,a}(x)=c_{\kappa,a}v_{\kappa,a}(x)\,dx,\quad
c_{\kappa,a}^{-1}=2a^{\lambda}\Gamma(\lambda+1),
\]
 and $\mathcal{F}_{\kappa,a}(f)(y)$
is the $(\kappa,a)$-generalized Fourier transform \eqref{eq1} on the real line.
Firstly, let us investigate when the kernel of $\mathcal{F}_{\kappa,a}$ is uniformly bounded.
Using \cite[Sect.~5]{SKO12}, we can write the kernel as
\begin{equation}\label{eq4}
B_{\kappa,a}(x,y)=j_{\lambda}\Bigl(\frac{2}{a}\,|xy|^{a/2}\Bigr)+
\frac{\Gamma(\lambda+1)}{\Gamma(\lambda+1+2/a)}\,\frac{xy}{(ai)^{2/a}}\,
j_{\lambda+\frac{2}{a}}\Bigl(\frac{2}{a}\,|xy|^{a/2}\Bigr),
\end{equation}
{where} $j_{\lambda}(x)=2^{\lambda}\Gamma(\lambda+1)x^{-\lambda}J_{\lambda}(x)$
is the normalized  Bessel function and  $J_{\lambda}(x)$ is the classical Bessel
function.
Then the asymptotic behavior  of  $J_{\lambda}(x)$ (see \cite[Chapt.~VII, 7.1]{Wa66})
 immediately allows us to derive
the following

\begin{proposition}\label{concon}
The conditions 
\begin{equation}\label{eq5}
0<a\le 2,\quad \kappa\ge\frac{1}{2}-\frac{a}{4},\quad\text{or}\quad a \ge 2,\quad \kappa\ge 0,
\end{equation}
are necessary and sufficient for boundedness  of the kernel $B_{\kappa,a}(x,y)$.

\end{proposition}

The main goal  of this section is to prove
Theorem \ref{thm2} for $d=1$.


\begin{proof}
Note that $B_{\kappa,a}(x,y)=b_{\kappa,a}(xy)$, where
\begin{equation}\label{eq6}
b_{\kappa,a}(x)=j_{\lambda}\Bigl(\frac{2}{a}\,|x|^{a/2}\Bigr)+
\frac{\Gamma(\lambda+1)}{\Gamma(\lambda+1+2/a)}\,\frac{x}{(ai)^{2/a}}\,
j_{\lambda+\frac{2}{a}}\Bigl(\frac{2}{a}\,|x|^{a/2}\Bigr).
\end{equation}
Therefore, under the conditions of Theorem \ref{thm2}, it suffices  to establish the inequality $|b_{\kappa,a}(x)|\le 1$ for $x\in\mathbb{R}$.

Let $a=\frac{2}{R}$, $R\in\mathbb{N}$.
Equality \eqref{eq6} can be written as
\[
b_{\kappa,a}(x)=j_{\lambda}\Bigl(R|x|^{1/R}\Bigr)+
\frac{\Gamma(\lambda+1)}{\Gamma(\lambda+1+R)}\Bigl(\frac{R}{2}\Bigr)^{R}(-i)^{R}\,x
j_{\lambda+R}\Bigl(R|x|^{1/R}\Bigr).
\]

Let $x\in\mathbb{R}$. In the case  $R=2r+1$, $a=\frac{2}{2r+1}$, $r\in\mathbb{Z}_{+}$, and $v=(2r+1)x^{\frac{1}{2r+1}}$, $x=\bigl(\frac{v}{2r+1}\bigr)^{2r+1}$, there holds
\begin{equation}\label{eq7}
e_{2r+1}(v,\lambda)=b_{\kappa,a}\Bigl(\Bigl(\frac{v}{2r+1}\Bigr)^{2r+1}\Bigr)=j_{\lambda}(v)+
i(-1)^{r+1}\,\frac{v^{2r+1}}{2^{2r+1}(\lambda+1)_{2r+1}}\,
j_{\lambda+2r+1}(v),
\end{equation}
where
\[
(a)_n=\frac{\Gamma(a+n)}{\Gamma(a)}=a(a+1)\cdots(a+n-1)
\]
is the Pochhammer symbol.

In the case $R=2r$, $a=\frac{1}{r}$, $r\in\mathbb{N}$, and $v=2r|x|^{\frac{1}{2r}}\sign{}x$, $x=\big(\frac{v}{2r}\big)^{2r}\sign{}v$,
we have
\begin{equation}\label{eq8}
e_{2r}(v,\lambda)=b_{\kappa,a}\Bigl(\Bigl(\frac{v}{2r}\Bigr)^{2r}\sign{}v\Bigr)=j_{\lambda}(v)+
(-1)^{r}\,\frac{v^{2r}}{2^{2r}(\lambda+1)_{2r}}\,j_{\lambda+2r}(v)\,\sign{}v.
\end{equation}
In order to see that $|e_{2r+1}(v,\lambda)|, |e_{2r}(v,\lambda)|\le 1$, we will need several auxiliary results.
We start with   
  the following identity
\begin{equation}\label{eq9}
\frac{v^2}{4(\lambda+1)(\lambda+2)}\,j_{\lambda+2}(v)=j_{\lambda+1}(v)-j_{\lambda}(v),
\end{equation}
which follows from the  recurrence relation for the Bessel function
$J_{\lambda}(v)$ (see \cite[Chapter~III, 3.2]{Wa66}). Then  by induction we establish

\begin{lemma}\label{lem1}
If $r\in\mathbb{N}$, then
\begin{equation}\label{eq10}
\frac{v^{2r}}{2^{2r}(\lambda+1)_{2r}}\,j_{\lambda+2r}(v)=
(-1)^{r}j_{\lambda}(v)+\sum_{s=1}^{r-1}(-1)^{s+r}
\binom{r}{s}\frac{(\lambda+r)_{s}}{(\lambda+1)_{s}}\,j_{\lambda+s}(v)
+\frac{(\lambda+r+1)_{r-1}}{(\lambda+1)_{r-1}}\,j_{\lambda+r}(v).
\end{equation}
\end{lemma}

\begin{proof} For  $r=1$, the needed formula coincides with  \eqref{eq9}.
Assume that \eqref{eq10} is valid for every $k\le r-1$ and  $\lambda$.
 Denote by
 $a_{s}^r(\lambda)$, $s=0,1,\dots,r$,
 the coefficients by
 $j_{\lambda+s}(v)$ in the decomposition  \eqref{eq10}.
  Taking into account  \eqref{eq9} and the inductive assumption, we derive that
\begin{align*}
\frac{v^{2r}}{2^{2r}(\lambda+1)_{2r}}\,j_{\lambda+2r}(v)&=
\frac{v^{2r-2}}{2^{2r-2}(\lambda+1)_{2r-2}}\,\frac{v^2}{4(\lambda+2r-1)(\lambda+2r)}\,j_{(\lambda+2r-2)+2}(v)
\\
&=\frac{v^{2r-2}}{2^{2r-2}(\lambda+1)_{2r-2}}\,\{j_{(\lambda+1)+2r-2}(v)-j_{\lambda+2r-2}(v)\}
\\
&=\frac{\lambda+2r-1}{\lambda+1}\sum_{s=1}^{r}a_{s-1}^{r-1}(\lambda+1)j_{\lambda+s}(v)-
\sum_{s=0}^{r-1}a_{s}^{r-1}(\lambda)j_{\lambda+s}(v).
\end{align*}

It is enough to show that
\[
a_{0}^{r}(\lambda)=-a_{0}^{r-1}(\lambda),\quad a_{r}^{r}(\lambda)=\frac{\lambda+2r-1}{\lambda+1}\,a_{r-1}^{r-1}(\lambda+1),
\]
\[
a_{s}^{r}(\lambda)=\frac{\lambda+2r-1}{\lambda+1}\,a_{s-1}^{r-1}(\lambda+1)-a_{s}^{r-1}(\lambda),\quad s=1,\dots,r-1.
\]
Indeed, using
the induction step, we have that
\[
a_{0}^{r}(\lambda)=
(-1)^r,\quad a_{r}^{r}(\lambda)=\frac{\lambda+2r-1}{\lambda+1}\,\frac{(\lambda+r+1)_{r-2}}{(\lambda+2)_{r-2}}=
\frac{(\lambda+r+1)_{r-1}}{(\lambda+1)_{r-1}},
\]
and, for $s=1,\dots,r-1$,
\begin{align*}
a_s^r(\lambda)&=(-1)^{s+r}\Bigl\{\frac{\lambda+2r-1}{\lambda+1}
\binom{r-1}{s-1}\frac{(\lambda+r)_{s-1}}{(\lambda+2)_{s-1}}+\binom{r-1}{s}\frac{(\lambda+r-1)_{s}}{(\lambda+1)_{s}}\Bigr\}
\\
&=(-1)^{s+r}\binom{r}{s}\frac{(\lambda+r)_{s}}{(\lambda+1)_{s}}\Bigl\{\frac{\lambda+2r-1}{\lambda+r+s-1}\,\frac{s}{r}+
\frac{\lambda+r-1}{\lambda+r+s-1}\,\frac{r-s}{r}\Bigr\}
\\
&=(-1)^{s+r}\binom{r}{s}\frac{(\lambda+r)_{s}}{(\lambda+1)_{s}},
\end{align*}
which completes the proof.
\end{proof}

\begin{lemma}\label{lem2}
For $r\in\mathbb{Z}_{+}$, we have
\[
\frac{v^{2r+1}}{2^{2r+1}(\lambda+1)_{2r+1}}\,j_{\lambda+2r+1}(v)=(-1)^{r+1}\sum_{s=0}^{r}(-1)^{s}
\binom{r}{s}\frac{(\lambda+r+1)_{s}}{(\lambda+1)_{s}}\,j_{\lambda+s}'(v).
\]
\end{lemma}

\begin{proof}

 Using Lemma~\ref{lem1} and the equality
\[
j_{\lambda}'(v)=-\frac{v}{2(\lambda+1)}\,j_{\lambda+1}(v),
\]
we derive that
\begin{multline*}
\frac{v^{2r+1}}{2^{2r+1}(\lambda+1)_{2r+1}}\,j_{\lambda+2r+1}(v)=\frac{v}{2(\lambda+1)}\,
\frac{v^{2r}}{2^{2r}(\lambda+2)_{2r}}\,j_{(\lambda+1)+2r}(v)
\\
=
\frac{v}{2(\lambda+1)}\Bigl\{(-1)^{r}j_{\lambda+1}(v)+\sum_{s=1}^{r-1}(-1)^{s+r}
\binom{r}{s}\frac{(\lambda+r+1)_{s}}{(\lambda+2)_{s}}\,j_{\lambda+1+s}(v)+\frac{(\lambda+r+2)_{r-1}}{(\lambda+2)_{r-1}}\,j_{\lambda+1+r}(v)\Bigr\}
\\
=
(-1)^{r+1}j_{\lambda}'(v)+\sum_{s=1}^{r-1}(-1)^{s+r+1}
\binom{r}{s}\frac{(\lambda+r+1)_{s}}{(\lambda+1)_s}\,j_{\lambda+s}'(v)-\frac{(\lambda+r+1)_{r}}{(\lambda+1)_r}\,j_{\lambda+r}'(v).
\end{multline*}
\end{proof}

Taking into account  \eqref{eq7}, \eqref{eq8}, Lemmas~\ref{lem1}, \ref{lem2}  and
\[
j_{\lambda}(v)=c_{\lambda}\int_{-1}^{1}(1-t^2)^{\lambda-1/2}e^{-ivt}\,dt,\quad
j_{\lambda}'(v)=-ic_{\lambda}\int_{-1}^{1}(1-t^2)^{\lambda-1/2}te^{-ivt}\,dt
\]
with
$ c_{\lambda}=\frac{\Gamma(\lambda+1)}{\sqrt{\pi}\Gamma(\lambda+1/2)}$, $\lambda>-1/2$ (see \cite[Chapt.~III, 3.3]{Wa66},
we arrive at the following integral representations of
the functions $e_{2r+1}(v,\lambda)$, and $e_{2r}(v,\lambda)$.

\begin{lemma}\label{lem3}
If $r\in\mathbb{Z}_{+}$, $\lambda>-1/2$, then
\begin{equation}\label{eq11}
e_{2r+1}(v,\lambda)=c_{\lambda}\int_{-1}^{1}(1-t^2)^{\lambda-1/2}q_{2r+1}(t,\lambda)e^{-ivt}\,dt,
\end{equation}
where $q_{2r+1}(t,\lambda)$ is a polynomial of degree  $2r+1$ with respect to $t$ given by
\[
q_{2r+1}(t,\lambda)=1+t\sum_{s=0}^{r}(-1)^{s}
\binom{r}{s}\frac{(\lambda+r+1)_{s}}{(\lambda+1/2)_{s}}\,(1-t^2)^s.
\]
\end{lemma}

\begin{lemma}\label{lem4}
If $r\in\mathbb{N}$, $\lambda>-1/2$, then
\begin{equation}\label{eq11----}
e_{2r}(v,\lambda)=c_{\lambda}\int_{-1}^{1}(1-t^2)^{\lambda-1/2}q_{2r}(t,\lambda)e^{-ivt}\,dt,
\end{equation}
where $q_{2r}(t,\lambda)$ is a polynomial of degree  $2r$ with respect to $t$ given by
\[
q_{2r}(t,\lambda)=q_{2r}(t,v,\lambda)
=1+\sign{}v\,\Bigl\{\sum_{s=0}^{r}(-1)^{s}
\binom{r}{s}\frac{(\lambda+r)_{s}}{(\lambda+1/2)_{s}}\,(1-t^2)^s\Bigr\}.
\]
\end{lemma}

 For our further analysis, it is important to know for which $\lambda$ the polynomials  $q_{2r}(t,\lambda)$, $q_{2r+1}(t,\lambda)$ are nonnegative on $[-1,1]$. If  $r=0$,
  $q_{1}(t,\lambda)=1-t\ge 0$ 
 on $[-1,1]$, that is,
 $q_{1}$ does not depend on $\lambda$. This special case corresponds to parameters $a=2,\, \kappa\ge 0$
and hence \eqref{eq11} is a well-known integral representation of the kernel of the one-dimensional Dunkl transform.
  In other cases, positivity conditions
 for $q_{2r}(t,\lambda)$ and $q_{2r+1}(t,\lambda)$ depend on  $\lambda$.
 We will see that there holds
 \begin{equation}\label{eq13}
q_{2r}(t,\lambda)=1+\sign v\,P_{2r}^{(\lambda-1/2)}(t),\quad q_{2r+1}(t,\lambda)=1+P_{2r+1}^{(\lambda-1/2)}(t),
\end{equation}
where
  $\{P_n^{(\alpha)}(t)\}_{n=0}^{\infty}$ is the system of Gegenbauer (ultraspherical) polynomials, that is, the family of polynomials orthogonal on $[-1,1]$  with respect to the weight function $(1-t^2)^{\alpha}$, $\alpha>-1$, normalized by $P_n^{(\alpha)}(1)=1$. Note that
\begin{equation}\label{eq12}
\frac{1}{\lambda}\,C_n^{\lambda}(t)=\frac{2\Gamma(2\lambda+n)}{n!\,\Gamma(2\lambda+1)}\,P_n^{(\lambda-1/2)}(t),\quad
\lambda>-1/2,
\end{equation}
where
$C_n^{\lambda}(t)$ are the Gegenbauer polynomials given in \cite[Chapt. X,10.9]{BE53a}.

\begin{lemma}\label{lem5}
For  $\lambda>-1/2$ and $r\ge 0$, there hold
\begin{equation*}
P_{2r}^{(\lambda-1/2)}(t)=\sum_{s=0}^{r}(-1)^{s}
\binom{r}{s}\frac{(\lambda+r)_{s}}{(\lambda+1/2)_{s}}\,(1-t^2)^s
\end{equation*}
and
\begin{equation*}
P_{2r+1}^{(\lambda-1/2)}(t)=t\sum_{s=0}^{r}(-1)^{s}
\binom{r}{s}\frac{(\lambda+r+1)_{s}}{(\lambda+1/2)_{s}}\,(1-t^2)^s.
\end{equation*}
Therefore, \eqref{eq13} holds.
\end{lemma}

 \begin{proof}
Since
$
\frac{(-r)_s}{s!}=(-1)^{s}\binom{r}{s},
 $ 
   taking into account  \eqref{eq12}, \cite[Chapt.~X, 10.9(21)]{BE53a}, we~get
\begin{multline*}
\sum_{s=0}^{r}(-1)^{s}\binom{r}{s}\frac{(\lambda+r)_{s}}{(\lambda+1/2)_{s}}\,(1-t^2)^s
=\sum_{s=0}^{r}\frac{(-r)_s(\lambda+r)_{s}}{s!\,(\lambda+1/2)_{s}}\,(1-t^2)^s\\
={}_{2}F_1(-r,\lambda+r;\lambda+1/2;1-t^2)=\frac{(2r)!\,\Gamma(2\lambda+1)}
{2\Gamma(2\lambda+2r)}\,\frac{1}{\lambda}\,C_{2r}^{\lambda}(t)=P_{2r}^{(\lambda-1/2)}(t).
\end{multline*}
Similarly, using \eqref{eq12}, \cite[Chapt. X, 10.9(22)]{BE53a}, we arrive at
\begin{multline*}
t\sum_{s=0}^{r}(-1)^{s}\binom{r}{s}\frac{(\lambda+r+1)_{s}}{(\lambda+1/2)_{s}}\,(1-t^2)^s
=t\sum_{s=0}^{r}\frac{(-r)_s(\lambda+r+1)_{s}}{s!\,(\lambda+1/2)_{s}}\,(1-t^2)^s\\
=t\,{}_{2}F_1(-r,\lambda+r+1;\lambda+1/2;1-t^2)=\frac{(2r+1)!\,\Gamma(2\lambda+1)}
{2\Gamma(2\lambda+2r+1)}\,\frac{1}{\lambda}\,C_{2r+1}^{\lambda}(t)=P_{2r+1}^{(\lambda-1/2)}(t).
\end{multline*}
 \end{proof}

We complete the proof of Theorem \ref{thm2} noting that
 $|P_{n}^{(\lambda-1/2)}(t)|\le 1$ for  $t\in [-1,1]$ under the condition  $\lambda\ge 0$ or, equivalently, $\kappa\ge 1/2$
(see \cite[Chapt.~VII, 7.32.2]{Se74}. Then in light of  Remark \ref{rem2} and Lemma \ref{lem5}),
 the polynomials  $q_{2r}(t,\lambda)$ and $q_{2r+1}(t,\lambda)$ are nonnegative on $[-1,1]$ and therefore Lemmas \ref{lem3} and \ref{lem4}, together with \eqref{eq7} and \eqref{eq8},
yield the statement of Theorem \ref{thm2}. To illustrate this, for $a=\frac{2}{2r+1}$, $v=(2r+1)x^{\frac{1}{2r+1}}$, we have
\[
|b_{\kappa,a}(x)|=|e_{2r+1}(v)|\le c_{\lambda}\int_{-1}^{1}(1-t^2)^{\lambda-1/2}q_{2r+1}(t,\lambda)\,dt
=1.
\]
\end{proof}

The following integral representations of $b_{\kappa,a}(x)$ follow from the lemmas above.
\begin{corollary}\label{cor1} 
If $\lambda=(2\kappa-1)/a$ and
$\kappa>\frac{1}{2}-\frac{a}{4}$,  then for $x\in\mathbb{R}$
\[
b_{\kappa,a}(x)=c_{\lambda}\int_{-1}^{1}(1-t^2)^{\lambda-1/2}(1+P_{2r+1}^{(\lambda-1/2)}(t))e^{-i(2r+1)x^{\frac{1}{2r+1}}{}t}\,dt,\quad
a=\frac{2}{2r+1},\quad r\in\mathbb{Z}_+,
\]
and
\[
b_{\kappa,a}(x)=c_{\lambda}\int_{-1}^{1}(1-t^2)^{\lambda-1/2}(1+\sign x\,P_{2r}^{(\lambda-1/2)}(t))e^{-i(2r|x|^{\frac{1}{2r}}\sign x)t}\,dt,\quad
a=\frac{1}{r},\quad r\in\mathbb{N}.
\]
\end{corollary}

\begin{remark}
The  representations of $b_{\kappa,a}(x)$ given in Corollary \ref{cor1} can be also obtained from 
\[
n!\int_{0}^{\pi}e^{iz\cos\theta}C_{n}^{\lambda}(\cos\theta)(\sin\theta)^{2\lambda}\,d\theta=
2^{\lambda}\sqrt{\pi}\,\Gamma(\lambda+1/2)(2\lambda)_n\,i^nz^{-\lambda}J_{\lambda+n}(z)
\]
(see \cite[Chapt. X, 10.9(38)]{BE53a}) without applying Lemmas~\ref{lem1} and \ref{lem2}. These lemmas are of independent interest.
\end{remark}

\begin{remark}
The positiveness of the polynomials $q_{2r}(t,\lambda)$, $q_{2r+1}(t,\lambda)$ is sufficient, but not necessary
for the estimate $|B_{\kappa,a}(x,y)|\le 1$ to hold.
In \cite{GIT16}, it was mentioned that for
$a=1$ this estimate holds also for  $1/4<\kappa_0\le \kappa<1/2$.
\end{remark}

Note that the case $\kappa<1/2$ corresponds to $\lambda<0$, which we study in detail in the next section.

\vspace{0.6mm}

\section{The case $d=1$ and $\lambda_{\kappa,a}<0$}\label{sec3}

Let us investigate whether  the polynomials  $q_{2r}(t,\lambda)$ and
$q_{2r+1}(t,\lambda)$ are nonnegative for  $\lambda=\lambda_{\kappa,a}=(2\kappa-1)/a<0$. In order to do this, we decompose
them by polynomials $q_{2r}(t,0)$ and $q_{2r+1}(t,0)$.

First, we decompose the Gegenbauer polynomials in terms of the Chebyshev polynomials using the well-known result  (see \cite[10.9(17)]{BE53a})
\begin{equation}\label{eq14}
C_{n}^{\lambda}(\cos\theta)=\sum_{s=0}^{n}\frac{(\lambda)_s(\lambda)_{n-s}}{s!\,(n-s)!}\cos{}(n-2m)\theta.
\end{equation}
Let us start with the case $R=2r+1$.
\begin{lemma}
For any  $r\in\mathbb{Z}_+$ and $\lambda>-1/2$, there holds
\begin{equation}\label{eq15}
P_{2r+1}^{(\lambda-1/2)}(t)=\sum_{s=0}^{r}b_s^r(\lambda)P_{2r+1-2s}^{(-1/2)}(t),
\end{equation}
where
\[
b_s^r(\lambda)=\frac{(2r+2-s)_{s}(\lambda)_{s}(\lambda+r+1)_{r-s}}{4^r(\lambda+1/2)_{r}s!},\quad s=0,1,\dots,r.
\]
\end{lemma}

\begin{proof} Taking into account   \eqref{eq12} and \eqref{eq14}, we obtain
\[
\frac{\Gamma(2\lambda+2r+1)}{(2r+1)!\,\Gamma(2\lambda)}\,P_{2r+1}^{(\lambda-1/2)}(t)=
C_{2r+1}^{\lambda}(t)=2\sum_{s=0}^{r}\frac{(\lambda)_s(\lambda)_{2r+1-s}}{s!\,(2r+1-s)!}\,P_{2r+1-2s}^{(-1/2)}(t).
\]
Hence,
the duplication formula for gamma function  implies
$$
b_s^r(\lambda)=\frac{2(2r+1)!\,(\lambda)_s\Gamma(2\lambda)\Gamma(\lambda+2r+1-s)}{(2r+1-s)!\,s!\,\Gamma(\lambda)\Gamma(2\lambda+2r+1)}
=\frac{(2r+2-s)_{s}(\lambda)_{s}(\lambda+r+1)_{r-s}}{4^r(\lambda+1/2)_{r}s!}.
$$
\end{proof}

\begin{remark}\label{rem2}
We see that in the decomposition \eqref{eq15} the zero coefficient is positive, and all other coefficients are also positive for $\lambda>0$ and
negative for $\lambda<0$. Note that the normalization of the Gegenbauer polynomials $P_{n}^{(\alpha)}(1) =1$  implies the equality
\[
\sum_{s=0}^{r}b_s^r(\lambda)=1
\]
and for $\lambda>0$ the Gegenbauer polynomials $P_{2r+1}^{(\lambda-1/2)}(t)$ are the convex hull of the Chebyshev polynomials
$P_{2r+1-2s}^{(-1/2)}(t)$, $s=0,1,\dots,r$. In particular, taking into account that 
 $P_{2r+1-2s}^{(-1/2)}(t)=\cos{}(({2r+1-2s})\arccos t)$, we easily obtain for the Gegenbauer polynomials the estimate
\[
|P_{2r+1}^{(\lambda-1/2)}(t)|\le 1,\qquad   t\in [-1,1]\quad\mbox{ and} \quad \lambda>0.
\]
\end{remark}
Thus,  we are in a position to state the required result on the decomposition for the polynomial $q_{2r+1}(t,\lambda)$.




\begin{corollary}\label{cor2}
For any $r\in\mathbb{Z}_+$ and $\lambda>-1/2$, there holds
\[
q_{2r+1}(t,\lambda)=\sum_{s=0}^{r}b_s^r(\lambda)q_{2r-2s+1}(t,0).
\]
\end{corollary}

Using \eqref{eq13} and
the properties of the coefficients $b_s^r(\lambda)$, $s=0,1,\dots,r$, from
Remark \ref{rem2}, we establish the following result.


\begin{corollary}\label{cor3}
For any $r\in\mathbb{N}$ and $\lambda\in (-1/2,0)$,
the polynomial $q_{2r+1}(t,\lambda)$ is negative at the points
of local minimum of the Chebyshev polynomial $P_{2r+1}^{(-1/2)}(t)=
\cos{}(({2r+1})\arccos t)$.
\end{corollary}

In the case $R=2r$,
similarly, one can obtain the decomposition of the polynomial $q_{2r}(t,\lambda)$.

\begin{lemma}
For any  $r\in\mathbb{N}$ and $\lambda>-1/2$, there holds 
\[
P_{2r}^{(\lambda-1/2)}(t)=\sum_{s=0}^{r}d_s^r(\lambda)P_{2r-2s}^{(-1/2)}(t),
\]
where
\[
d_s^r(\lambda)=\frac{2(2r+1-s)_{s}(\lambda)_{s}(\lambda+r)_{r-s}}{4^r(\lambda+1/2)_{r}s!},\quad
s=0,1,\dots,r-1,\quad
d_r^r(\lambda)=\frac{(r+1)_{r}(\lambda)_{r}}{4^r(\lambda+1/2)_{r}r!}.
\]
\end{lemma}

\medskip
\begin{proof}
In light of  \eqref{eq12} and \eqref{eq14}, we have
\[
\frac{\Gamma(2\lambda+2r)}{(2r)!\,\Gamma(2\lambda)}\,P_{2r}^{(\lambda-1/2)}(t)=
C_ {2r}^{\lambda}(t)=2\sum_{s=0}^{{r-1}}\frac{(\lambda)_s(\lambda)_{2r-s}}{s!\,(2r-s)!}\,P_{2r-2s}^{(-1/2)}(t)+\frac{((\lambda)_r)^2}{(r!)^2}\,P_{0}^{(-1/2)}(t).
\]
Then, using
the duplication formula for gamma function, we obtain for  $s=0,1,\dots,r-1$
\[
d_s^r(\lambda)=
\frac{2(2r)!\,(\lambda)_s\Gamma(2\lambda)\Gamma(\lambda+2r-s)}{(2r-s)!\,s!\,\Gamma(\lambda)\Gamma(2\lambda+2r)}
=\frac{2(2r+1-s)_{s}(\lambda)_{s}(\lambda+r)_{r-s}}{4^r(\lambda+1/2)_{r}s!},
\]
while for  $s=r$,
\[
d_r^r(\lambda)=\frac{(2r)!\,(\lambda)_r\Gamma(2\lambda)\Gamma(\lambda+r)}{(r!)^{2}\Gamma(\lambda)\Gamma(2\lambda+2r)}
=\frac{(r+1)_{r}(\lambda)_{r}}{4^r(\lambda+1/2)_{r}r!}.
\]
\end{proof}

\begin{corollary}\label{cor4}
For any  $r\in\mathbb{N}$ and $\lambda>-1/2$, the following decomposition holds
\[
q_{2r}(t,\lambda)=\sum_{s=0}^{r}d_s^r(\lambda)q_{2r-2s}(t,0).
\]
\end{corollary}

Since $ d_0^r(\lambda)>0$ and  $\sign d_s^r(\lambda)=\sign\lambda$, $s=1,\dots,r$,
for $\lambda>-1/2$, taking into account
\eqref{eq13},
we arrive at the following result.

\begin{corollary}\label{cor5}
If $r\in\mathbb{N}$ and $\lambda\in (-1/2,0)$, then the polynomial
$q_{2r}(t,\lambda)$ is negative at the points of local extremum of the
Chebyshev polynomial $P_{2r}^{(-1/2)}(t)=\cos{}(2r\arccos t)$.
\end{corollary}

Corollaries \ref{cor3} and \ref{cor5} show that   the polynomials
$q_{2r}(t,\lambda)$ and $q_{2r+1}(t,\lambda)$ for $r\ge 1$ change sign  and,
therefore, the method of the proof of the estimate $|B_{\kappa,a}(x,y)|\le 1$
used in Theorem \ref{thm2}  cannot be applied when $\lambda<0$. In this
case the problem remains open.

\subsection{On positive definiteness}\label{subsec3-1}
Recall that a continuous function $f$ on $\mathbb{R}$ is positive definite if for any
 $\{v_s\}_{s=1}^{n}\subset\mathbb{R}$ and $\{z_s\}_{s=1}^{n}\subset\mathbb{C}$
there holds
\[
\sum_{s,l=1}^{n}z_s\overline{z_{l}}f(v_s-v_l)\ge 0.
\]
Bochner's theorem states that any continuous positive definite function $f(v)$, $f(0)=1$, is the Fourier transform of a probability measure.

Recall that the function $e_{2r+1}(v,\lambda)$ is given in \eqref{eq7} and \eqref{eq11}, \eqref{eq13}.  Taking into account \eqref{eq8}, \eqref{eq11----} and \eqref{eq13},
let us define also the functions
\[
e_{2r,+}(v,\lambda)=c_{\lambda}\int_{-1}^{1}(1-t^2)^{\lambda-1/2}(1+P_{2r}^{(\lambda-1/2)}(t))\,e^{-ivt}\,dt,
\]
\[
e_{2r,-}(v,\lambda)=c_{\lambda}\int_{-1}^{1}(1-t^2)^{\lambda-1/2}(1-P_{2r}^{(\lambda-1/2)}(t))\,e^{-ivt}\,dt.
\]

Combining Theorem \ref{thm2} for $d=1$, Lemma~\ref{lem5}, Corollaries~\ref{cor2},
\ref{cor3}, \ref{cor4}, and \ref{cor5}, we deduce  the following

\begin{corollary}
The functions  $e_{2r+1}(v,\lambda)$ and $e_{2r,\pm}(v,\lambda)$ 
are positive definite  if and only if $\lambda\ge 0$.
\end{corollary}

It is worth mentioning that  our proof of Theorem \ref{thm2} for $d=1$, in fact,  was
based on the positive definiteness of these functions.

\vspace{0.6mm}

\section{The kernel of the generalized Fourier transform in the multivariate case}

\subsection{When  the kernel is bounded by one}\label{subsec4-1}
In this section,
for the sake of completeness,
 we give the proof of Theorem~\ref{thm2} in the general case, following ideas from the paper \cite{CBL18}.
We stress that the proof below also allows one to deal with the case $d=1$.


Let $d\in\mathbb{N}$, $a>0$,  $\langle\kappa\rangle\ge 0$,
and $\eta=\lambda_{\kappa}=\langle\kappa\rangle+\frac{d-2}{2}>0$.
For $w\ge 0$ and $\tau\in [-1,1]$ define
\begin{equation}\label{eq16}
\Psi_{a}^d(w,\tau)=2^{2\eta/a}\Gamma\Bigl(\frac{2\eta+a}{a}\Bigr)\sum_{j=0}^{\infty}e^{-\frac{i\pi j}{a}}\,\frac{\eta+j}{\eta}\,w^{-2\eta/a}J_{\frac{2(\eta+j)}{a}}(w)C_{j}^{\eta}(\tau).
\end{equation}
Putting in \eqref{eq16}
\[
x,y\in\mathbb{R}^d,\quad x=|x|x',\quad y=|y|y',\quad x',y'\in\mathbb{S}^{d-1},\quad w=\frac{2}{a}\,(|x||y|)^{a/2},\quad \tau=\langle x',\,y'\rangle,
\]
we obtain the function
\begin{equation*}
K_{a}^d(x,y)=a^{2\eta/a}\Gamma\Bigl(\frac{2\eta+a}{a}\Bigr)(|x||y|)^{-\eta}\sum_{j=0}^{\infty}e^{-\frac{i\pi j}{a}}\,\frac{\eta+j}{\eta}\,J_{\frac{2(\eta+j)}{a}}\Bigl(\frac{2}{a}(|x||y|)^{a/2}\Bigr)C_{j}^{\eta}(\langle x',\,y'\rangle).
\end{equation*}
Note that $K_{a}^d(x,y)=K_{a}^d(y,x)$ and $\Psi_{a}^d(0,\tau)=K_{a}^d(0,y)=1$.

Let $V_{\kappa}f(x)$ be the interwining operator in the Dunkl harmonic analysis \cite{Ro02}, which is a positive operator satisfying
\[
V_{\kappa}f(x)=\int_{\mathbb{R}^d}f(\xi)\,d\mu_{x}^{\kappa}(\xi).
\]
The representing measures $\mu_{x}^{\kappa}(\xi)$ are compactly supported probability
measures with $\text{supp}\,\mu_{x}^{\kappa}(\xi)\subset\text{co}{}\{gx\colon g\in G\}$
\cite{Ro99}.

The kernel of the generalized Fourier transform $\mathcal{F}_{\kappa,a}$ is given by \cite[Chaps. 4, 5]{SKO12}
\begin{multline}\label{eq17}
B_{\kappa,a}(x,y)=V_{\kappa}K_{a}^d(x,|y|\cdot)(y')\\
=a^{2\eta/a}\Gamma\Bigl(\frac{2\eta+a}{a}\Bigr)(|x||y|)^{-\eta}\sum_{j=0}^{\infty}e^{-\frac{i\pi j}{a}}\,\frac{\eta+j}{\eta}\,J_{\frac{2(\eta+j)}{a}}\Bigl(\frac{2}{a}(|x||y|)^{a/2}\Bigr)V_{\kappa}C_{j}^{\eta}(\langle x',{\cdot}\,\rangle)(y').
\end{multline}
If $\langle \kappa\rangle=0$, that is, $\eta=(d-2)/2$, then $V_{\kappa}=\text{I}$ is the identity operator and $K_{a}^d(x,y)=B_{0,a}(x,y)$.

Since the operator $V_{\kappa}$ is positive and $V_{\kappa}1=1$, the proof of equality \eqref{eq2} 
 can proceed as follows:
if  $|\Psi_{a}^d(w,\tau)|\le 1$, then $|B_{\kappa,a}(x,y)|\le 1$.

Let further $\frac{2}{a}\in\mathbb{N}$, $a=\frac{2}{R}$. The inequality
$|\Psi_{a}^d(w,\tau)|\le 1$ was proved in \cite{CBL18} under the conditions $\eta>0$ and
$\eta=(d-2)/2$.
We note that the proof of this  estimate in \cite{CBL18}  
also holds for $\eta=\langle \kappa\rangle+(d-2)/2>0$. Hence,
\eqref{eq2} is valid for $\langle \kappa\rangle+(d-2)/2>0$. If $\eta=0$, then
$d=1$, $\kappa=1/2$ or $d=2$, $\kappa\equiv0$ and in these cases
\eqref{eq2} holds as well (see Theorem \ref{thm2}  and \cite{BIE13}).
%
This completes the proof of  Theorem~\ref{thm2}.

\subsection{Positive definiteness of  $\Psi_{a}^d$.
Another proof of Theorem~\ref{thm2}}

In \cite{CBL18}, to evaluate the kernel $B_{0,a}(x,y)$, the authors used the Laplace transform of the function $\Psi_{a}^d$ given by  \eqref{eq16}.
In the general case  $d\ge 1$, we  give another proof of the estimate $|\Psi_{a}^d|\le 1$ based on the approach from the papers  \cite{BL20,DD21}, which allows one to show positive definiteness of the function $\Psi_{a}^d(w,\tau)$ with respect to $w$. Note  also that for $a=1,2$ the functions 
\begin{equation*}
\Psi_{1}^d(w,\tau)=j_{\lambda_{\kappa}-1/2}(w\sqrt{(1+\tau)/2}\,), \quad \Psi_{2}^d(w,\tau)=e^{-iw\tau},
\end{equation*}
are positive definite
(see \cite[Example 4.18]{SKO12}, \cite{CBL18}).



Let
\[
I_{\eta}(b)=\Bigl(\frac{b}{2}\Bigr)^{\eta}\,\sum_{j=0}^{\infty}\frac{1}{j!\,\Gamma(j+\eta+1)}\Bigl(\frac{b^2}{4}\Bigl)^j
\]
be the modified Bessel function of the first kind (see \cite[Chapt. 7, 7.2.2]{BE53a}) and\[
\Phi_2^{(m)}(\beta_1,\dots,\beta_m;\gamma;x_1,\dots,x_m)=
\sum_{j_1,\dots,j_m\ge 0}\frac{(\beta_1)_{j_1}\cdots (\beta_m)_{j_m}}{(\gamma)_{j_1+\cdots+j_m}}\,\frac{x_1^{j_1}}{j_{1}!}\cdots\frac{x_m^{j_m}}{j_{m}!}
\]
be the second Humbert function of $m$ variables (see \cite[Chapt. 2, 2.1.1.2]{Ex83}).
In the case when $\gamma-\sum_{j=1}^m\beta_j>0$ and  $\beta_j>0$, $j=1,\dots,m$, 
the hypergeometric
function $\Phi_2^{(m)}$ admits the following integral representation
\begin{equation}\label{eq18}
\Phi_2^{(m)}(\beta_1,\dots,\beta_m;\gamma;x_1,\dots,x_m)
=C_{\beta}^{(\gamma)}\int_{T^m}e^{\sum_{j=1}^mx_jt_j}\Bigl(1-\sum_{j=1}^{m}t_j\Bigr)^{\gamma-\sum_{j=1}^m\beta_j-1}
\prod_{j=1}^{m}t_j^{\beta_j-1}\,dt_1\cdots dt_m,
\end{equation}
where
\[
C_{\beta}^{(\gamma)}=\frac{\Gamma(\gamma)}{\Gamma(\gamma-\sum_{j=1}^{m}\beta_j)\prod_{j=1}^{m}\Gamma(\beta_j)}
\]
and
\[
T^m=\Bigl\{(t_1,\dots,t_m)\colon t_j\ge 0,\ j=1,\dots,m,\ \sum_{j=1}^{m}t_j\le 1\Bigr\}
\]
is the unit simplex in $\mathbb{R}^m$ (\cite{CW09,Hu20}).

Taking into account the results from
 \cite{BL20,DD21}, we obtain the following proposition, which gives an alternative proof of Theorem~\ref{thm2}.

\begin{proposition}\label{thm3}
Let $d\in\mathbb{N}$, $a=\frac{2}{R}$, $R\in\mathbb{N}$, $R\ge 2$, $\eta=\langle \kappa\rangle+\frac{d-2}{2}>0$, $\tau\in [-1,1]$ and $q=\arccos \tau$. The function  $w \mapsto\Psi_{a}^d(w,\tau)$ is a positive definite entire function of exponential type
\begin{equation}\label{eq19}
\theta(a,\tau)=
\begin{cases}
\cos\frac{q}{R}, & R=2r\ \text{and}\ 0\le q\le\pi \ \text{or}\ R=2r+1\ \text{and}\ 0\le q\le\pi/2,\\
\cos\frac{\pi-q}{R}, & R=2r+1\ \text{and}\ \pi/2\le q\le\pi.
\end{cases}
\end{equation}
\end{proposition}

\begin{proof} Let us consider the function
\[
 f_{R,\eta}(b,\tau)=\Gamma(R\eta+1)\Bigl(\frac{b}{2}\Bigr)^{\eta}\sum_{j=0}^{\infty}
 \frac{j+\eta}{\eta}\,I_{R(j+\eta)}(b)C_{j}^{\eta}(\tau), \quad R\in\mathbb{N}.
\]
Let $\tau=\langle x',\,y'\rangle$, $b=-iw$, $R=\frac{2}{a}$, $R\ge 2$. Since $I_{\eta}(-iw)=e^{-i\frac{\pi\eta}{2}}J_{\eta}(w)$ \cite[Chapt. 7, 7.2.2]{BE53a}, then $f_{R,\eta}(b,\tau)=\Psi_{a}^d(w,\tau)$. It is known (see \cite{BL20,DD21}) that
\[
f_{R,\eta}(b,\tau)=b_0\Phi_{2}^{(R-1)}(\eta,\dots,\eta;R\eta;b_1-b_0,\dots,b_{R-1}-b_0),
\]
where $b_j=b\cos{}(\frac{q-2\pi j}{R})$, $j=0,\dots,R-1$.
In light of \eqref{eq18} and using $\eta>0$, we get
\begin{multline}\label{eq20}
\Psi_{a}^d(w,\tau)=f_{R,\eta}(-iw,\tau)\\
=C_{\eta}\int_{T^{R-1}}e^{-iw\{\cos{}(\frac{q}{R})+\sum_{j=1}^{R-1}(\cos{}(\frac{q-2\pi j}{R})-\cos{}(\frac{q}{R}))t_j\}}\Bigl(1-\sum_{j=1}^{R-1}t_j\Bigr)^{\eta-1}\prod_{j=1}^{R-1}t_j^{\eta-1}\,
\,dt_1\cdots dt_{R-1},
\end{multline}
where $C_{\eta}=\Gamma(R\eta)/(\Gamma(\eta))^R$.
For $\alpha>0$ and $\beta\in\mathbb{R}$, the  function $\alpha e^{i\beta w}$ is positive definite, hence the function $\Psi_{a}^d(w,\tau)$ is also positive definite  with respect to $w$.

The representation \eqref{eq20} implies  that $\Psi_a^d(w,\tau)$ is an entire function of exponential type
with respect to $w$. Let us calculate its type. Since for $j=1,\dots,R-1$, $q=\arccos\tau$, $\tau\in [-1,1]$,
$
\cos\frac{q}{R}-\cos\frac{q-2\pi j}{R}=2\cos\frac{\pi i}{R}\cos\frac{\pi j-q}{R}\ge 0,
$ 
 then for any $(t_1,\dots,t_{R-1})\in T^{R-1}$,
\[
\min_{1\le j\le R-1}\cos\frac{q-2\pi j}{R}\le \cos\frac{q}{R}+\sum_{j=1}^{R-1}\Bigl(\cos\frac{q-2\pi j}{R}-\cos\frac{q}{R}\Bigr)t_j\le \cos\frac{q}{R}.
\]
Since
\begin{equation*}
\min_{1\le j\le R-1}\cos\frac{q-2\pi j}{R}=
\begin{cases}-\cos\frac{q}{R}, & R=2r\ \text{or}\ R=2r+1\ \text{and}\ 0\le q\le\pi/2,\\
-\cos\frac{\pi-q}{R}, & R=2r+1\ \text{and}\ \pi/2\le q\le\pi,
\end{cases}
\end{equation*}
the type of the function $\Psi_{a}^d(w,\tau)$ is given by \eqref{eq19}.
\end{proof}

\begin{remark}\label{rem3}
Using \eqref{eq16} and Lemmas~\ref{lem1}, \ref{lem2} we get for $t,\tau\in (-1,1)$, $\lambda=2\eta/a$,
\begin{equation*}
\Psi_{a}^d(w,\tau)=c_{\lambda}\int_{-1}^{1}(1-t^2)^{\lambda-1/2}\sum_{j=0}^{\infty}\frac{\eta+j}{\eta}\,P_{2j/a}^{(\lambda-1/2)}(t)
C_{j}^{\eta}(\tau)e^{-iwt}\,dt.
\end{equation*}
It is clear that the condition
\begin{equation*}
\psi_a^d(t,\tau)=\sum_{j=0}^{\infty}\frac{\eta+j}{\eta}\,P_{2j/a}^{(\lambda-1/2)}(t)
C_{j}^{\eta}(\tau)\ge 0,\quad t,\,\tau\in (-1,1),
\end{equation*}
is equivalent to  positive definiteness of the function $\Psi_{a}^d(w,\tau)$.
Proposition~\ref{thm3} shows that the function $\psi_{a}^d(w,\tau)$ is positive and its support as a function of $t$ lies on the interval $I_a=[-\theta(a,\tau),\theta(a,\tau)]$.
If $a=1$, then $R=2$, $\lambda=2\eta$ and $\theta(a,\tau)=\cos\frac{q}{2}=\sqrt{(1+\tau)/2}$. Using \cite[Chapt. VIII, 8.7(31)]{Ba54}, we obtain
\begin{align*}
\psi_{1}^d(t,\tau)&=(2\pi c_{\lambda})^{-1}(1-t^2)^{1/2-2\eta}\int_{-\infty}^{\infty}j_{\eta-1/2}(w\sqrt{(1+\tau)/2})e^{iwt}\,dw
\\
&=c(\eta)(1-t^2)^{1/2-2\eta}(1+\tau)^{1/2-\eta}(1+\tau-2t^2)^{\eta-1}\chi_{I_1}(t)\\
&=c(\eta)(1-t^2)^{1/2-2\eta}(1+\tau)^{1/2-\eta}(\tau-P_2^{(-1/2)}(t))^{\eta-1}\chi_{I_1}(t),
\end{align*}
where $\chi_{I_1}(t)$ is the characteristic function of the interval $I_1$ and $P_2^{(-1/2)}(t)=
\cos{}(2\arccos t)$ is the Chebyshev polynomial of degree two.
\end{remark}

\subsection{Parameters when the kernel is not bounded by one}\label{subsec4-2}
We will show that in some cases  the uniform norm of the kernel $B_{\kappa,a}(x,y)$
is not bounded by $1$ and 
 can be either finite or infinite.
 We mention that the conditions when the norm is not finite are known only in the one-dimensional case. 
In particular, if $d=1$, $0<a<2$, and $\kappa<\frac{1}{2}-\frac{a}{4}$, then $\|B_{\kappa,a}\|_{\infty}=\infty$ (see Proposition~\ref{concon}).


\begin{proof}[Proof of Theorem~\ref{thm4}] Let first $d=1$ and $\lambda=\lambda_{\kappa,a}=\frac{2\kappa-1}{a}$. Changing variables $x=(\frac{a}{2}|v|)^{2/a}\sign{}v$
in \eqref{eq6}, we get
\begin{multline}\label{eq21}
e_{\kappa,a}(v)=b_{\kappa,a}\Bigl(\Bigl(\frac{a}{2}\,|v|\Bigr)^{2/a}\sign{}v\Bigr)
=j_{\lambda}(v)+\frac{\Gamma(\lambda+1)\cos\frac{\pi}{a}}{2^{2/a}\Gamma(\lambda+1+2/a)}\,
|v|^{2/a}j_{\lambda+\frac{2}{a}}(v)\sign{}v
\\
{}-i\,\frac{\Gamma(\lambda+1)\sin\frac{\pi}{a}}{2^{2/a}\Gamma(\lambda+1+2/a)}\,
|v|^{2/a}j_{\lambda+\frac{2}{a}}(v)\sign{}v.
\end{multline}
If $0<a\le 1$ and $\kappa=\frac{1}{2}-\frac{a}{4}$, then $\lambda=-1/2$, $j_{-1/2}(v)=\cos v$, $j_{1/2}(v)=\frac{\sin v}{v}$.

Assume first that $\cos\frac{\pi}{a}=0$ or $a=\frac{2}{2r+1}$, $r\in\mathbb{N}$. Since $\sin\frac{\pi}{a}\neq 0$, $j_{-1/2}(2\pi)=1$, $j_{1/2}(2\pi)=0$, and $j_{2r+1/2}(2\pi)\neq 0$ (see \cite[Chapt. 15, 15.28]{Wa66}), then from \eqref{eq21} we get $|e_{\kappa,a}(2\pi)|=|e_{2r+1}(2\pi,-1/2)|>1$.

Let now  $\cos{}(\pi/a)\neq 0$. In light of \eqref{eq21}, there holds 
\[
|e_{\kappa,a}(v)|\ge \Bigl|\cos v+\frac{\Gamma(\frac{1}{2})\cos\frac{\pi}{a}}{2^{2/a}\Gamma(\frac{2}{a}+\frac{1}{2})}\,
|v|^{2/a}j_{\frac{2}{a}-\frac{1}{2}}(v)\sign{}v\Bigr|.
\]
Since
\[
(2\pi s)^{2/a}j_{\frac{2}{a}-\frac{1}{2}}(2\pi s)=\frac{2^{2/a}\Gamma(\frac{2}{a}+\frac{1}{2})}{\Gamma(\frac{1}{2})}\,
\Bigl\{\cos\frac{\pi}{a}+O\Bigl(\frac{1}{s}\Bigr)\Bigr\}
\]
as $s\to +\infty$
(see \cite[Chapt.~VII, 7.1]{Wa66}), we deduce that for sufficiently large $s$
\[
|e_{\kappa,a}(2\pi s)|\ge 1+\cos^2\frac{\pi}{a}+O\Bigl(\frac{1}{s}\Bigr)>1,
\]
which completes the proof of  \eqref{eq3}. Note that similar arguments also implies the proof of \eqref{eq3} for  $1<a<2$ and $\kappa=\frac{1}{2}-\frac{a}{4}$.

\smallbreak
Now we consider the multivariate case. Suppose that  $d\ge 1$, $a\in (1,2)\cup (2,\infty)$, $\langle\kappa\rangle\ge 0$, and $\eta=\lambda_{\kappa}$, $\lambda=\frac{2\eta}{a}=\frac{2\langle\kappa\rangle+d-2}{a}$. In view of \eqref{eq5} and using estimate \eqref{eq3},  we may assume that  $\eta,\lambda>-1/2$ and $\cos\frac{\pi}{a}\neq 0$ for $d\ge 1$.

Note that  \eqref{eq17} can be equivalently written as 
\[
B_{\kappa,a}(x,y)
=\sum_{j=0}^{\infty}e^{-\frac{i\pi j}{a}}\,\frac{\eta+j}{\eta}\,\frac{\Gamma(\lambda+1)}{2^{2j/a}\,\Gamma(\lambda+1+2j/a)}\,w^{2j/a}
j_{\lambda+\frac{2j}{a}}(w)V_{\kappa}C_{j}^{\eta}(\langle x',{\cdot}\,\rangle)(y').
\]
Since
\[
|j_{\lambda}(w)|\le 1,\quad \Bigl|\frac{\eta+j}{\eta}\,V_{\kappa}C_j^{\eta}(\langle x',{\cdot}\,\rangle)(y')\Bigr|\le c(\eta)(j^{2\eta}+1),\quad \frac{1}{\eta}\,C_1^{\eta}(t)=2t,
\]
then for $0\le w\le 1$ we have
\begin{multline*}
\Bigl|\sum_{j=2}^{\infty}e^{-\frac{i\pi j}{a}}\,\frac{\eta+j}{\eta}\,\frac{\Gamma(\lambda+1)}{2^{2j/a}\Gamma(\lambda+1+2j/a)}\,w^{2j/a}
j_{\lambda+\frac{2j}{a}}(w)V_{\kappa}C_{j}^{\eta}(\langle
x',{\cdot}\,\rangle)(y')\Bigr|\\
{}\le c(\eta)w^{4/a}\sum_{j=2}^{\infty}\frac{\Gamma(\lambda+1)(j^{2\eta}+1)}{2^{2j/a}\Gamma(\lambda+1+2j/a)}
\le c(a,\kappa)w^{4/a}
\end{multline*}
and
\begin{multline*}
B_{\kappa,a}(x,y)=j_{\lambda}(w)+\frac{2(\eta+1)\Gamma(\lambda+1)\cos\frac{\pi}{a}}{2^{2/a}\Gamma(\lambda+1+2/a)}\,w^{2/a}
j_{\lambda+\frac{2}{a}}(w)V_{\kappa}(\langle x',{\cdot}\,\rangle)(y')\\
{}-i\,\frac{2(\eta+1)\Gamma(\lambda+1)\sin\frac{\pi}{a}}{2^{2/a}\Gamma(\lambda+1+2/a)}\,w^{2/a}
j_{\lambda+\frac{2}{a}}(w)V_{\kappa}(\langle x',{\cdot}\,\rangle)(y')+O(w^{4/a}),\quad w\to 0.
\end{multline*}
Therefore,
\begin{equation}\label{eq22}
|B_{\kappa,a}(x,y)|\ge \Bigl|j_{\lambda}(w)+\frac{2(\eta+1)\Gamma(\lambda+1)\cos\frac{\pi}{a}}{2^{2/a}\Gamma(\lambda+1+2/a)}\,w^{2/a}
j_{\lambda+\frac{2}{a}}(w)V_{\kappa}(\langle x',{\cdot}\,\rangle)(y')+O(w^{4/a})\Bigr|.
\end{equation}

Since the operator $V_{\kappa}$
is isomorphism on the space of homogeneous polynomials of degree $1$, then for some
  $x',y'\in\mathbb{S}^{d-1}$ we deduce
\begin{equation}\label{eq23}
V_{\kappa}(\langle x',{\cdot}\,\rangle)(y')=\int_{\mathbb{R}^d}\langle x',\xi\rangle\,d\mu_{y'}^{\kappa}(\xi)\neq 0,\quad \sign (V_{\kappa}(\langle x',{\cdot}\,\rangle)(y'))=\sign{}(\cos{}(\pi/a)).
\end{equation}
We have
\[
j_{\lambda}(w)=1+O(w^2),\quad j_{\lambda+\frac{2}{a}}(w)=1+O(w^2)
\]
as $w\to 0$.
This, \eqref{eq22} and \eqref{eq23} imply that, for
given $x',y'\in\mathbb{S}^{d-1}$ and sufficiently small positive~$w$,
\[
|B_{\kappa,a}(x,y)|\ge 1+\frac{2(\eta+1)\Gamma(\lambda+1)|\!\cos\frac{\pi}{a}|}{2^{2/a}\Gamma(\lambda+1+2/a)}\,w^{2/a}
|V_{\kappa}(\langle x',{\cdot}\,\rangle)(y')|+O(w^{4/a})+O(w^2)>1.
\]
\end{proof}

\begin{remark}\label{rem4}
(i) For $0<a\le 1$, $\frac{2}{a}\in\mathbb{N}$, equality \eqref{eq2}  holds provided  $\kappa\ge\kappa_0(a)$, where  $\frac{1}{2}-\frac{a}{4}<\kappa_0(a)\le 1/2$.
The problem of determining  $\kappa_0(a)$ is open.

\smallbreak
(ii) Theorem \ref{thm4} shows that our conjecture in \cite{GIT16} asserting that
\eqref{eq3} holds under the condition $2\langle\kappa\rangle+d+a\ge3$
is not valid for $d\ge 1$ and $a\in (1,2)\cup (2,\infty)$.
\end{remark}

\vspace{0.6mm}

\section{Image of the Schwartz space}\label{sec5}

In this section, for the $(\kappa,a)$-generalized Fourier transform, we study the image of the Schwartz space $\mathcal{S}(\mathbb{R}^d)$.
\subsection{The case of $a>0$}
Let $d\nu_{\eta}(u)=b_{\eta}u^{2\eta+1}\,du$, $b_{\eta}^{-1}=2^{\eta}\Gamma(\eta+1)$,
$d\nu_{\eta,a}(u)=b_{\eta,a}u^{2\eta+a-1}du$, $b_{\eta,a}^{-1}=a^{2\eta/a}\Gamma(2\eta/a+1)$,
\[
H_{\eta}(f_0)(v)=\int_{0}^{\infty}f_0(u)j_{\eta}(uv)\,d\nu_{\eta}(u),\quad \eta\ge-1/2,
\]
be the Hankel transform and
\[
H_{\eta,a}(f_0)(v)=\int_{0}^{\infty}f_0(u)j_{\frac{2\eta}{a}}\Bigl(\frac{2}{a}\,(vu)^{a/2} \Bigr)\,d\nu_{\eta,a}(u),\quad 2\eta+a\ge 1,
\]
be $a$-deformed Hankel transform (see \cite{SKO12,GIT16}).

Recall that $\lambda_{\kappa}=\langle \kappa\rangle+(d-2)/2$, $\lambda=\lambda_{\kappa,a}=2\lambda_{\kappa}/a$.

\begin{example}\label{exa1}
Consider  $f(x)=e^{-|x|^2}\in\mathcal{S}(\mathbb{R}^d)$, $f(x)=f_0(\rho)$, $\rho=|x|$.
If $|y|=v$, $\rho=(a/2)^{1/a}u^{2/a}$, then (see \cite{SKO12,GIT16})
\begin{align*}
\mathcal{F}_{\kappa,a}(f)(y)=H_{\lambda_{\kappa},a}(f_0)(v)&=
\int_{0}^{\infty}f_0(\rho)j_{\frac{2\lambda_{\kappa}}{a}}\Bigl(\frac{2}{a}\,(v\rho)^{a/2} \Bigr)\,d\nu_{\lambda_{\kappa},a}(\rho)\\
&=\int_{0}^{\infty}\exp \Bigl(-\Bigl(\frac{a}{2}\Bigr)^{2/a}u^{4/a}\Bigr)j_{\lambda}
\Bigl(\Bigl(\frac{2}{a}\Bigr)^{1/2}\,v^{a/2}u\Bigr)\,d\nu_{\lambda}(u).
\end{align*}
Let
\[
g_a(u)=\exp \Bigl(-\Bigl(\frac{a}{2}\Bigr)^{2/a}u^{4/a}\Bigr).
\]
Assuming that  $\mathcal{F}_{\kappa,a}(f)(y)$ is rapidly decreasing at infinity, the same is true for the function
\[
G_a(v)=H_{\lambda}(g_a)(v)=\int_{0}^{\infty}g_a(u)j_{\lambda}(uv)\,d\nu_{\lambda}(u)
\]
as $v\to\infty$. If $a=4$, then in light of \cite[Chap.~VIII, 8.6(4)]{Ba54} the function
\[
G_a(v)=\frac{c_{\lambda_{\kappa},a}}{(1+v^2)^{\lambda+3/2}}
\]
decreases at infinity not faster than a power function.
If $2/a$ is a non-integer and  $a\neq
4$, then  $4/a$ is also non-integer, and therefore $g_a(u)$
has finite smoothness at the origin.
On the other hand, since  $g_e,G_a\in L^1(\mathbb{R}_+,d\nu_{\lambda})$, we obtain
\[
g_e(u)=H_{\lambda}(G_a)(u)=\int_{0}^{\infty}G_a(v)j_{\lambda}(uv)\,d\nu_{\lambda}(v)
\]
and the right-hand side is infinitely differentiable at the origin due to fast decreasing of $G_a(v)$.
This contradiction shows us that the generalized Fourier transform  cannot rapidly decrease at infinity.
Moreover, let $\partial f=f'$. Since
\[
(\partial_v^2j_{\lambda}(uv))\bigr|_{v=0}=-\frac{u^2}{2(\lambda+1)},\quad (\partial_v^4j_{\lambda}(uv))\bigr|_{v=0}=\frac{3u^4}{4(\lambda+1)(\lambda+2)},
\]
then 
\[
G_a(v)=\xi_0+\xi_1v^2+O(v^4),\qquad v\to 0,\qquad
 \xi_1\neq 0,
\]
and
\[
\mathcal{F}_{\kappa,a}(f)(y)=\xi_0+\widetilde{\xi}_1|y|^a+O(|y|^{2a}),\qquad y\to 0,\qquad \widetilde{\xi}_1\neq 0.
\]
If $a$ is not  even, then $\mathcal{F}_{\kappa,a}(f)(y)$ has finite smoothness at the origin.
\end{example}

The following statement follows from Example \ref{exa1}.

\begin{proposition}\label{lem6}
Let $d\in\mathbb{N}$.

\smallbreak
\textup{(i)} \
The condition $\frac{a}{2}\in\mathbb{N}$ is necessary for the embedding $\mathcal{F}_{\kappa,a}(\mathcal{S}(\mathbb{R}^d))\subset C^{\infty}(\mathbb{R}^d)$ to hold.

\smallbreak
\textup{(ii)} \
The condition $\frac{2}{a}\in\mathbb{N}$ is necessary for the set $\mathcal{F}_{\kappa,a}(\mathcal{S}(\mathbb{R}^d))$ to consist of rapidly decreasing functions at infinity.

\end{proposition}

\begin{remark}\label{rem2---}
We see that $\mathcal{F}_{\kappa,a}(\mathcal{S}(\mathbb{R}^d))=\mathcal{S}(\mathbb{R}^d)$ only for $a=2$. Hence, the claim  in \cite[Lemma~2.12]{Jo16} that the  Schwartz space is  invariant
under the generalized Fourier transform is false for $a\neq 2$.
\end{remark}

The conditions in Proposition~\ref{lem6} are also sufficient at least in the one-dimensional case. Indeed, suppose $\lambda=(2\kappa-1)/a\ge -1/2$ and denote the
even and odd parts of $f$,
as usual, by
\[
f_{e}(x)=\frac{f(x)+f(-x)}{2},\quad f_{o}(x)=\frac{f(x)-f(-x)}{2}.
\]
Then the one-dimensional generalized Fourier transform  can be written as
\begin{equation}\label{eq24}
\mathcal{F}_{\kappa,a}(f)(y)=2\int_{0}^{\infty}(B_{\kappa,a}(\,{\cdot}\,,y))_{e}(x) f_e(x)\,d\mu_{\kappa,a}(x)+
2\int_{0}^{\infty}(B_{\kappa,a}(\,{\cdot}\,,y))_{o}(x) f_o(x)\,d\mu_{\kappa,a}(x).
\end{equation}
Putting in  \eqref{eq24}
\begin{equation}\label{eq25}
\begin{gathered}
x=x(u)=\Bigl(\frac{a}{2}\Bigr)^{\frac{1}{a}}\,u^{\frac{2}{a}},\quad  g(u)=Af(u)=f(x(u))=f\Bigl(\Bigl(\frac{a}{2}\Bigr)^{\frac{1}{a}}\,u^{\frac{2}{a}}\Bigr),\\
g_e(u)=Af_e(u),\quad g_o(u)=Af_o(u),\quad x,u\in\mathbb{R}_+,
\end{gathered}
\end{equation}
and taking into account  \eqref{eq4}, we deduce
\begin{multline}\label{eq26}
\mathcal{F}_{\kappa,a}(f)(y)=\int_{0}^{\infty}Af_e(u)j_{\lambda}\Bigl(\Bigl(\frac{2}{a}\Bigr)^{1/2}|y|^{a/2}u\Bigr)\,d\nu_{\lambda}(u)
\\
{}+e^{-\frac{\pi i}{a}}\Bigl(\frac{2}{a}\Bigr)^{1/a}y
\int_{0}^{\infty}u^{-2/a}Af_o(u)j_{\lambda+\frac{2}{a}}\Bigl(\Bigl(\frac{2}{a}\Bigr)^{1/2}|y|^{a/2}u\Bigr)\,d\nu_{\lambda+\frac{2}{a}}(u)
\\
=H_{\lambda}(g_e)\Bigl(\Bigl(\frac{2}{a}\Bigr)^{1/2}|y|^{a/2}\Bigr)+
e^{-\frac{\pi i}{a}}\Bigl(\frac{2}{a}\Bigr)^{1/a}yH_{\lambda+\frac{2}{a}}(u^{-2/a}g_o)\Bigl(\Bigl(\frac{2}{a}\Bigr)^{1/2}|y|^{a/2}\Bigr).
\end{multline}

If $f\in \mathcal{S}(\mathbb{R})$, then $g_e(u)$, $u^{-2/a}g_o(u)$ decrease rapidly at infinity and there exist
\[
\lim_{u\to 0+0}u^{-2/a}g_o(u)=\Bigl(\frac{a}{2}\Bigr)^{\frac{1}{a}}f'(0).
\]
Therefore, the representation \eqref{eq26} shows that, for even $a$ we have the embedding
$\mathcal{F}_{\kappa,a}(\mathcal{S}(\mathbb{R}))\subset C^{\infty}(\mathbb{R})$.

Further, if $\frac{2}{a}\in\mathbb{N}$, $f\in\mathcal{S}(\mathbb{R})$, then the functions  $g_e(u),u^{-2/a}g_o(u) \in\mathcal{S}(\mathbb{R}_+)$ as well as (see \cite{Ro02}) the even functions
\[
 H_{\lambda}(g_e)(v),\quad H_{\lambda+\frac{2}{a}}(u^{-2/a}g_o)(v)\in \mathcal{S}(\mathbb{R}).
\]

Hence, we prove the following

\begin{proposition}\label{lem9}
Suppose $\frac{2}{a}\in\mathbb{N}$, $\lambda\ge-1/2$, $f\in\mathcal{S}(\mathbb{R})$, then
\begin{equation}\label{eq27}
\mathcal{F}_{\kappa,a}(f)(y)=F_1\bigl(|y|^{a/2}\bigr)+yF_2\bigl(|y|^{a/2}\bigr),
\end{equation}
where the even functions  $F_1,F_2\in\mathcal{S}(\mathbb{R})$.
\end{proposition}

Thus, for $2/a\in\mathbb{N}$ the generalized Fourier transform decreases rapidly at infinity
and we arrive at the following result.

\begin{proposition}
\textup{(i)} \
The embedding $\mathcal{F}_{\kappa,a}(\mathcal{S}(\mathbb{R}))\subset C^{\infty}(\mathbb{R})$ is valid if and only if $\frac{a}{2}\in\mathbb{N}$.

\smallbreak
\textup{(ii)} \
The set $\mathcal{F}_{\kappa,a}(\mathcal{S}(\mathbb{R}))$ consists of rapidly decreasing functions at infinity  if and only if $\frac{2}{a}\in\mathbb{N}$.
\end{proposition}

Let
\[
\mathcal{S}_{org}(\mathbb{R})=\{f\in\mathcal{S}(\mathbb{R})\colon \partial^{k}f(0)=0,\ k\in\mathbb{N}\}.
\]

The class $\mathcal{S}_{org}(\mathbb{R})$ is dense in  $L^{2}(\mathbb{R},d\mu_{\kappa,a})$.
Suppose  $a>0$, $f(x)\in\mathcal{S}_{org}(\mathbb{R})$, then the functions $g_e(u),u^{-2/a}g_o(u) \in\mathcal{S}(\mathbb{R}_+)$, cf. \eqref{eq26}. Therefore, the even functions  $H_{\lambda}(g_e)(v)$, $H_{\lambda+\frac{2}{a}}(u^{-2/a}g_o)(v)$ also belong  to $\mathcal{S}(\mathbb{R})$.

Thus, we obtain the following

\begin{proposition}
Suppose $a>0$, $\lambda\ge-1/2$, and $f\in\mathcal{S}_{org}(\mathbb{R})$; then the generalized Fourier transform  $\mathcal{F}_{\kappa,a}(f)$ enjoys the representation  \eqref{eq27}.
\end{proposition}

\subsection{The case of irrational $a$}
We will show that if  $a$ is irrational, any nontrivial Schwartz function possesses similar  properties to the Gaussian function
  in Example~\ref{exa1}. To see this, we need auxiliary  properties of the kernel of the generalized Fourier transform.

Let $\mathbb{S}^{d-1}$ be the unit sphere in $\mathbb{R}^{d}$, $x=\rho x'$, $\rho=|x|\in\mathbb{R}_+$, $x'\in\mathbb{S}^{d-1}$, and $dx'$ be the Lebesgue measure on the sphere. If $a_{\kappa}^{-1}=\int_{\mathbb{S}^{d-1}}v_{\kappa}(x')\,dx'$,\,\, $d\sigma_{\kappa}(x')=a_{\kappa}v_{\kappa}(x')\,dx'$, then $d\mu_{\kappa,a}(x)=d\nu_{\lambda_{\kappa},a}(\rho)\,d\sigma_{\kappa}(x')$ and $c_{\kappa,a}=b_{\kappa,a}a_{\kappa}$.

Denote by $\mathcal{H}_{n}^{d}(v_{\kappa})$ the subspace of $\kappa$-spherical harmonics
of degree $n\in \mathbb{Z}_{+}$ in $L^{2} (\mathbb{S}^{d-1},d\sigma_{\kappa})$ (see
\cite[Chap.~5]{DunXu01}). Let $\mathcal{P}_{n}^{d}$ be the space of homogeneous
polynomials of degree $n$ in $\mathbb{R}^{d}$. Then $\mathcal{H}_{n}^{d}(v_{\kappa})$ is the
restriction of $\ker \Delta_{\kappa}\cap \mathcal{P}_{n}^{d}$ to the sphere
$\mathbb{S}^{d-1}$.

If $l_{n}$ is the dimension of $\mathcal{H}_{n}^{d}(v_{\kappa})$, we denote by
$\{Y_{n}^{j}\colon j=1,\ldots,l_{n}\}$ the real-valued orthonormal basis
$\mathcal{H}_{n}^{d}(v_{\kappa})$ in $L^{2}(\mathbb{S}^{d-1},d\sigma_{\kappa})$.
 A union of these bases forms orthonormal basis in
$L^{2}(\mathbb{S}^{d-1},d\sigma_{\kappa})$ consisting of
$k$-spherical harmonics.

Let us rewrite \eqref{eq17} as follows
\[
B_{\kappa,a}(x,y)
=\sum_{j=0}^{\infty}e^{-\frac{i\pi j}{a}}\,\frac{\lambda_{\kappa}+j}{\lambda_{\kappa}}\,\frac{\Gamma(2\lambda_{\kappa}/a+1)(|x||y|)^{j}}{a^{2j/a}\,\Gamma(2(\lambda_{\kappa}+j)/a+1)}\,
j_{\frac{2(\lambda_{\kappa}+j)}{a}}\Bigl(\frac{2}{a}\,(|x||y|)^{a/2}\Bigr)V_{\kappa}C_{j}^{\lambda_{\kappa}}(\langle x',{\cdot}\,\rangle)(y').
\]
Integrating this and using orthogonality of $\kappa$-spherical harmonics, as in the case of Dunkl kernel
 (see \cite[Theorem 5.3.4]{DunXu01}, \cite[Corollary 2.5]{Ro03}), we obtain the following crucial
 property of the kernel of the generalized Fourier transform.

\begin{proposition}\label{lem7}
If $x,y\in\mathbb{R}^d$, $x=\rho x'$, $y=vy'$, then
\[
\int_{\mathbb{S}^{d-1}}B_{\kappa,a}(x,vy')Y_{n}^{j}(y')\,d\sigma_{\kappa}(y')=\frac{e^{-\frac{i\pi n}{a}}\Gamma(2\lambda_{\kappa}/a+1)}{a^{2n/a}
\,\Gamma(2(\lambda_{\kappa}+n)/a+1)}\,v^n
j_{\frac{2(\lambda_{\kappa}+n)}{a}}\Bigl(\frac{2}{a}\,(\rho v)^{a/2}\Bigr)Y_n^j(x).
\]
\end{proposition}

Denote  by $\mathcal{S}(\mathbb{R}_+)$ the subspace of even functions from  $\mathcal{S}(\mathbb{R})$.

\begin{proposition}
For irrational  $a$ and a nontrivial function $f\in\mathcal{S}(\mathbb{R}^d)$, we have 
$\mathcal{F}_{\kappa,a}(f)\notin\mathcal{S}(\mathbb{R}^d)$.
\end{proposition}

\begin{proof} \textbf{1.} First, assume that  $f(x)=\rho^n\psi(\rho)Y_n^j(x')$, where $n\in\mathbb{Z}_+$, $x=\rho x'$, and $\psi\in\mathcal{S}(\mathbb{R}_+)$.
Since  $\rho^nY_n^j(x')=Y_n^j(x)$ is the homogeneous polynomial of degree $n$, then $f(x)=\psi(|x|)Y_n^j(x)\in\mathcal{S}(\mathbb{R}^d)$.
If $y=vy'$, $\rho=(a/2)^{1/a}u^{2/a}$, then by \cite{GIT16}
\begin{multline*}
\mathcal{F}_{\kappa,a}(f)(y)=e^{-\frac{i\pi n}{a}}Y_n^j(y)H_{\lambda_{\kappa}+n,a}(\psi)(v)
=e^{-\frac{i\pi n}{a}}Y_n^j(y)\int_{0}^{\infty}\psi(\rho)j_{\frac{2(\lambda_{\kappa}+n)}{a}}\Bigl(\frac{2}{a}\,(v\rho)^{a/2} \Bigr)\,d\nu_{\lambda_{\kappa},a}(\rho)\\
=e^{-\frac{i\pi n}{a}}Y_n^j(y)\int_{0}^{\infty}\psi\Bigl(\Bigl(\frac{a}{2}\Bigr)^{1/a}u^{2/a}\Bigr)
j_{\frac{2(\lambda_{\kappa}+n)}{a}}\Bigl(\Bigl(\frac{2}{a}\Bigr)^{1/2}\,v^{a/2}u \Bigr)\,d\nu_{\frac{2(\lambda_{\kappa}+n)}{a}}(u).
\end{multline*}
Moreover, following the ideas used in Example~\ref{exa1},
we obtain that the function
\[
g_a(v)=\int_{0}^{\infty}\psi\Bigl(\Bigl(\frac{a}{2}\Bigr)^{1/a}u^{2/a}\Bigr)j_{\frac{2(\lambda_{\kappa}+n)}{a}}\Bigl(\Bigl(\frac{2}{a}\Bigr)^{1/2}\,v^{a/2}u \Bigr)\,d\nu_{\frac{2(\lambda_{\kappa}+n)}{a}}(u)
\]
with irrational $a$ cannot decrease rapidly as $v\to\infty$ and it has finite smoothness at the origin.
Therefore, it follows that  
 $\mathcal{F}_{\kappa,a}(f)\notin \mathcal{S}(\mathbb{R}^d)$.

\smallbreak
\textbf{2.}
Let now $f\in\mathcal{S}(\mathbb{R}^d)$ be any non-zero function. Then
its spherical $\kappa$-harmonic expansion is given by
\[
f(\rho x')=\sum_{n=0}^{\infty}\sum_{j=1}^{l_{n}}f_{nj}(\rho)Y_{n}^{j}(x'),\quad f_{nj}(\rho)=\int_{\mathbb{S}^{d-1}}f(\rho x')Y_{n}^{j}(x')\,d\sigma_{\kappa}(x')
\]
(see \cite{GIT16}). Since the subspaces  $\mathcal{H}_n^d(v_{\kappa})$ are orthogonal, then $f_{nj}^{(j)}(0)=0$, $j=0,1,\dots,n-1$.
Changing variables  $x'\to -x'$ implies
$f_{nj}(-\rho)
=(-1)^nf_{nj}(\rho).
$
Hence,
setting a non-zero function $g_{nj}(f)(x)=f_{nj}(\rho)Y_{n}^{j}(x')\in\mathcal{S}(\mathbb{R}^d)$,
we have $g_{nj}(x)=\rho^n\psi(\rho)Y_{n}^{j}(x')$ with some $\psi\in\mathcal{S}(\mathbb{R}_+)$. 

 In order to apply the  results obtained in case 1, we need to show that
$$\mathcal{F}_{\kappa,a}(g_{nj}(f))(y)=g_{nj}(\mathcal{F}_{\kappa,a}(f))(y).$$ Indeed, we note that 
\[
\mathcal{F}_{\kappa,a}(g_{nj}(f))(y)=e^{-i\pi n/a}Y_n^j(y)H_{\lambda_{\kappa}+n,a}(\psi)(v).
\]
On the other hand, in view of  Proposition~\ref{lem7}, we deduce that 
\begin{align*}
g_{nj}(\mathcal{F}_{\kappa,a}(f))(y)&=Y_{n}^{j}(y')\int_{\mathbb{S}^{d-1}}\mathcal{F}_{\kappa,a}(f)(vy')Y_{n}^{j}(y')\,d\sigma_{\kappa,a}(y')\\
&=e^{-\frac{i\pi n}{a}}Y_{n}^{j}(y) \int_{\mathbb{R}^d}\frac{f(x)\Gamma(2\lambda_{\kappa}/a+1)}{a^{2n/a}
\,\Gamma(2(\lambda_{\kappa}+n)/a+1)}
j_{\frac{2(\lambda_{\kappa}+n)}{a}}\Bigl(\frac{2}{a}\,(\rho v)^{a/2}\Bigr)Y_n^j(x)\,d\mu_{\kappa,a}(x) \\
&=e^{-\frac{i\pi n}{a}}Y_{n}^{j}(y)\int_{0}^{\infty}\psi(\rho)
j_{\frac{2(\lambda_{\kappa}+n)}{a}}\Bigl(\frac{2}{a}\,(\rho v)^{a/2}\Bigr)\,d\nu_{\lambda_{\kappa}+n,a}(\rho)\\
&=e^{-\frac{i\pi n}{a}}Y_n^j(y)H_{\lambda_{\kappa}+n,a}(\psi)(v).
\end{align*}
Assuming here that  $f,\mathcal{F}_{\kappa,a}(f)\in\mathcal{S}(\mathbb{R}^d)$ yields  $g_{nj}(f),\mathcal{F}_{\kappa,a}(g_{nj}(f))\in\mathcal{S}(\mathbb{R}^d)$, which contradicts the first case.
\end{proof}

Summarizing, the generalized Fourier  transform for irrational $a$  drastically  deforms even very smooth functions.
It was mentioned
in  \cite[Chap. 5]{SKO12} that the generalized Fourier  transform has a finite order only for rational $a$.
 Therefore, the case of irrational $a$ is of little  interest in harmonic analysis.


\vspace{0.6mm}

\section{Non-deformed unitary transforms generated by $\mathcal{F}_{\kappa,a}$}\label{sec6}

Let us study the case  $a=\frac{2}{2r+1}$ and  $\lambda=\lambda_{\kappa,a}=(2\kappa-1)/a\ge -1/2$ in more detail.
Recall that $A$ is given by \eqref{eq25} and
we can also assume  that $x, u\in \mathbb{R}$.

Since
\[
\int_{-\infty}^{\infty}|f(x)|^2\,d\mu_{\kappa,a}(x)=\int_{-\infty}^{\infty}|Af(u)|^2\,d\widetilde{\nu}_{\lambda}(u),\qquad
d\widetilde{\nu}_{\lambda}(u)
=\frac{|u|^{2\lambda+1}\,du}{2^{\lambda+1} \Gamma(\lambda+1)},
\]
the linear operator $A\colon L^{2}(\mathbb{R},d\mu_{\kappa,a})\to L^{2}(\mathbb{R},d\widetilde{\nu}_{\lambda})$ is an isometric isomorphism. The inverse operator is given by
$A^{-1}g(x)=g((2r+1)^{1/2}\,x^{1/(2r+1)})$.

In view of \eqref{eq7} and \eqref{eq25}, $B_{\kappa,a}(x,y)=e_{2r+1}(uv,\lambda)$ and
\begin{equation*}
A\mathcal{F}_{\kappa,a}(f)(v)=\int_{\mathbb{R}}e_{2r+1}(uv,\lambda)Af(u)
\,d\widetilde{\nu}_{\lambda}(u).
\end{equation*}

This formula defines the \textit{non-deformed} transform $\mathcal{F}_{r}^{\lambda}$, for
 $\lambda> -1/2$ and $r\in\mathbb{Z}_+$,
\begin{align*}
\mathcal{F}_{r}^{\lambda}(g)(v)&=\int_{-\infty}^{\infty}e_{2r+1}(uv,\lambda)g(u)
\,d\widetilde{\nu}_{\lambda}(u)\\
&=\int_{-\infty}^{\infty}\Bigl(j_{\lambda}(uv)+i(-1)^{r+1}\,\frac{(uv)^{2r+1}}{2^{2r+1}(\lambda+1)_{2r+1}}\,
j_{\lambda+2r+1}(uv)\Bigr)g(u)
\,d\widetilde{\nu}_{\lambda}(u)\\
&=c_{\lambda}\int_{-\infty}^{\infty}\int_{-1}^{1}(1-t^2)^{\lambda-1/2}(1+P_{2r+1}^{(\lambda-1/2)}(t))\,e^{-iuvt}\,dt\,g(u)
\,d\widetilde{\nu}_{\lambda}(u).
\end{align*}
Moreover, its  kernel satisfies the estimate  $|e_{2r+1}(uv,\lambda)|\le M_{\lambda}<\infty$ and, importantly,   $M_{\lambda}=1$ for $\lambda\ge0$. If $r=0$, we recover the one-dimensional Dunkl transform.

Below we study  an invariant subspaces ($\subset C^{\infty}$) of the $\mathcal{F}_{r}^{\lambda}$ transform.
The Plancherel theorem for $\mathcal{F}_{\kappa,a}$ given by
\[
\int_{-\infty}^{\infty}|\mathcal{F}_{\kappa,a}(f)(y)|^2\,d\mu_{\kappa,a}(y)=\int_{-\infty}^{\infty}|f(x)|^2\,d\mu_{\kappa,a}(x), \quad f\in L^{2}(\mathbb{R},d\mu_{\kappa,a}),
\]
implies
that $\mathcal{F}_{r}^{\lambda}$ is  a unitary operator in $L^{2}(\mathbb{R},d\widetilde{\nu}_{\lambda})$, i.e.,
\[
\int_{-\infty}^{\infty}|\mathcal{F}_{r}^{\lambda}(g)(v)|^2\,d\widetilde{\nu}_{\lambda}(v)
=\int_{-\infty}^{\infty}|g(u)|^2\,d\widetilde{\nu}_{\lambda}(u).
\]
Since  the reverse operator satisfies $(\mathcal{F}_{\kappa,a})^{-1}(f)(x)=\mathcal{F}_{\kappa,a}(f)(-x)$ \cite[Theorem 5.3]{SKO12},
we have \begin{equation*}
(\mathcal{F}_{r}^{\lambda})^{-1}(f)(u)=\int_{-\infty}^{\infty}\overline{e_{2r+1}(uv,\lambda)}f(v)
\,d\widetilde{\nu}_{\lambda}(v).
\end{equation*}
If $g, \mathcal{F}_{r}^{\lambda}(g)\in L^{1}(\mathbb{R},d\widetilde{\nu}_{\lambda})$, then one may assume that  $g, \mathcal{F}_{r}^{\lambda}(g)\in C_b(\mathbb{R})$. Moreover,
the   inversion formula
\begin{equation}\label{eq28}
g(u)=\int_{-\infty}^{\infty}\overline{e_{2r+1}(uv,\lambda)}\mathcal{F}_{r}^{\lambda}(g)(v)\,d\widetilde{\nu}_{\lambda}(v)
\end{equation}
holds not only in $L_2$ sense
but also pointwise.

Considering the derivatives of the kernel $e_{2r+1}(uv,\lambda)$, we note that
\[
\partial_{v}^{n}e_{2r+1}(uv,\lambda)=(iu)^{n}c_{\lambda}\int_{-1}^{1}t^{n}
(1-t^2)^{\lambda-1/2}(1+P_{2r+1}^{(\lambda-1/2)}(t))\,e^{-iuvt}\,dt,
\]
and so,
\[
|\partial_{v}^{n}e_{2r+1}(uv,\lambda)|\le M_{\lambda}|u|^{n}.
\]
Then, for $g\in \mathcal{S}(\mathbb{R})$, we have
\begin{equation}\label{eq29}
|\partial^{n}\mathcal{F}_{r}^{\lambda}(g)(v)|=\Bigl|\int_{-\infty}^{\infty}g(u)\partial_{v}^{n}e_{2r+1}(uv,\lambda)
\,d\widetilde{\nu}_{\lambda}(u)\Bigr|\le M_{\lambda}\int_{-\infty}^{\infty}|u|^{n}|g(u)|\,d\widetilde{\nu}_{\lambda}(u)<\infty.
\end{equation}
Therefore, $\mathcal{F}_{r}^{\lambda}(\mathcal{S}(\mathbb{R}))\subset C^{\infty}(\mathbb{R})$. However, $\mathcal{F}_{r}^{\lambda}(\mathcal{S}(\mathbb{R}))\not\subset \mathcal{S}(\mathbb{R})$. Indeed, assuming $g,\mathcal{F}_{r}^{\lambda}(g)\in \mathcal{S}(\mathbb{R})$, by
orthogonality of the Gegenbauer polynomials for $s=0,1,\dots,r-1$,
\[
\int_{-1}^{1}t^{2s+1}
(1-t^2)^{\lambda-1/2}(1\pm P_{2r+1}^{(\lambda-1/2)}(t,\lambda))\,dt=0
\]
and
\begin{equation}\label{eq30}
\partial^{2s+1}\mathcal{F}_{r}^{\lambda}(g)(0)=0,\quad
\partial^{2s+1}g(0)=0,
\end{equation}
which is not true for arbitrary $g$.

Put for $n\in\mathbb{Z}_+$
\[
\mathcal{S}_{n}(\mathbb{R})=\{g\in\mathcal{S}(\mathbb{R})\colon \partial^{2s+1}g(0)=0,\quad s=0,1,\dots,n-1\},\quad \mathcal{S}_{0}(\mathbb{R})=\mathcal{S}(\mathbb{R}).
\]
The set  $\mathcal{S}_{n}(\mathbb{R})$ is dense in $L^{2}(\mathbb{R},d\mu_{\kappa,a})$ and in  $L^{2}(\mathbb{R},d\widetilde{\nu}_{\lambda})$.

\begin{example}
Consider the function
$g_{2s+1}(u)=u^{2s+1}e^{-u^2}\in\mathcal{S}_{s}(\mathbb{R})$,
$s\in\mathbb{Z}_+$. By means of \eqref{eq4}, \eqref{eq7}, \cite[Chap.~VIII, 8.6(14)]{Ba54},
and \cite[Chap.~VI, 6.1]{BE53a}, we get
\begin{align*}
\mathcal{F}_{r}^{\lambda}(g_{2s+1})(v)&=i(-1)^{r+1}v^{-\lambda}\int_{0}^{\infty}u^{\lambda+2s+2}e^{-u^2}J_{\lambda+2r+1}(uv)\,du\\
&=ic_{r,\lambda,s}\,v^{2r+1}\sum_{l=0}^{\infty}\frac{(\lambda+s+r+2)_l}{(\lambda+2r+2)_l}\Bigl(-\frac{v^2}{4}\Bigr)^{l}\\
&=ic_{r,\lambda,s}\,v^{2r+1}\Phi\Bigl(\lambda+s+r+2,\lambda+2r+2,-\frac{v^2}{4}\Bigr),\quad c_{r,\lambda,s}>0.
\end{align*}
We consider two cases. If $s=0,1,\dots,r-1$, then asymptotics as
$v\to\infty$ \cite[Chap.~VI,6.13.1]{BE53} implies
\[
\Phi\Bigl(\lambda+s+r+2,\lambda+2r+2,-\frac{v^2}{4}\Bigr)=\frac{\Gamma(\lambda+2r+2)}
{\Gamma(r-s)}\Bigl(\frac{v^2}{4}\Bigr)^{-(\lambda+s+r+2)}\Bigl(1+O\Bigl(\frac{1}{v^2}\Bigr)\Bigr),
\]
that is, $\mathcal{F}_{r}^{\lambda}(g_{2s+1})\notin\mathcal{S}(\mathbb{R})$.
If $s\ge r$, then applying the Kummer transform \cite[Chap.~VI,6.3.7)]{BE53} gives us
\begin{align}\label{eq31}
\mathcal{F}_{r}^{\lambda}(g_{2s+1})(v)&=ic_{r,\lambda,s}\,v^{2r+1}e^{-v^2/4}\Phi\Bigl(r-s,\lambda+2r+2,\frac{v^2}{4}\Bigr)\nonumber
\\
&=ic_{r,\lambda,s}\,v^{2r+1}e^{-v^2/4}
\sum_{l=0}^{s-r}\frac{(r-s)_l}{(\lambda+2r+2)_l}\Bigl(\frac{v^2}{4}\Bigr)^{l},
\end{align}
that is,
$\mathcal{F}_{r}^{\lambda}(g_{2s+1})(v)\in \mathcal{S}_{r}(\mathbb{R})$.
\end{example}

Since $g_{2r+1}\in \mathcal{S}_{r}(\mathbb{R})$ and $\mathcal{F}_{r}^{\lambda}(g_{2r+1})(v)\in \mathcal{S}_{r}(\mathbb{R})$,
one can conjecture that 
$\mathcal{F}_{r}^{\lambda}(\mathcal{S}_{r}(\mathbb{R}))=
\mathcal{S}_{r}(\mathbb{R})$.
In order to show this (see Proposition \ref{lem11}), we will need some auxiliary results.

Recall that for the weight function $|x|^{2\lambda+1}$
the differential-difference
Dunkl operator of the first and second order are given by
\[
T_{\lambda+1/2}g(u)=\partial g(u)+(\lambda+1/2)\,\frac{g(u)-g(-u)}{u},
\]
\[
\Delta_{\lambda+1/2}g(u)=T_{\lambda+1/2}^2g(u)=\partial^{2}
g(u)+\frac{2\lambda+1}{u}\,\partial g(u)-(\lambda+1/2)\,\frac{g(u)-g(-u)}{u^2}
\]
(see \cite{SKO12,Ro02}). Let us define the operator
\[
\delta_{\lambda}g(u)=T_{\lambda+1/2}^2g(u)-2r(\lambda+r+1)\,\frac{g(u)-g(-u)}{u^2},
\]
which  is obtained by changing  variables $x=x(u)$ as in \eqref{eq25} in the Dunkl Laplacian
\[
\Delta_{\kappa}f(x)=\frac{(2r+1)^{2r-1}}{u^{4r}}\,\delta_{\lambda}g(u).
\]
By direct calculations we verify that the kernel $e_{2r+1}(uv,\lambda)$ is the eigenfunction of $\delta_{\lambda}$:
\begin{equation}\label{eq32}
(\delta_{\lambda})_ue_{2r+1}(uv,\lambda)=-|v|^2e_{2r+1}(uv,\lambda).
\end{equation}
Using \cite[Proposition~2.18]{Ro02} for $g\in \mathcal{S}(\mathbb{R})$, we have
\begin{equation}\label{eq33}
\int_{-\infty}^{\infty}T_{\lambda+1/2}^2g(u)e_{2r+1}(uv,\lambda)\,d\widetilde{\nu}_{\lambda}(u)=
\int_{-\infty}^{\infty}g(u)(T_{\lambda+1/2}^2)_ue_{2r+1}(uv,\lambda)\,d\widetilde{\nu}_{\lambda}(u).
\end{equation}

Suppose $g\in \mathcal{S}_1(\mathbb{R})$, then
\[
\frac{g(u)-g(-u)}{u^2}=\frac{2g_{o}(u)}{u^2}\in \mathcal{S}(\mathbb{R})
\]
and
\[
\int_{-\infty}^{\infty}\frac{g(u)-g(-u)}{u^2}\,e_{2r+1}(uv,\lambda)\,d\widetilde{\nu}_{\lambda}(u)=
\int_{-\infty}^{\infty}g(u)\,\frac{e_{2r+1}(uv,\lambda)-e_{2r+1}(-uv,\lambda)}{u^2}\,d\widetilde{\nu}_{\lambda}(u).
\]
This, \eqref{eq32} and \eqref{eq33} for any $g\in \mathcal{S}_1(\mathbb{R})$ yield
\begin{align}\label{eq34}
\int_{-\infty}^{\infty}\delta_{\lambda}g(u)e_{2r+1}(uv,\lambda)\,d\widetilde{\nu}_{\lambda}(u)&=
\int_{-\infty}^{\infty}g(u)(\delta_{\lambda})_ue_{2r+1}(uv,\lambda)\,d\widetilde{\nu}_{\lambda}(u)
\nonumber
\\
&=-|v|^2\int_{-\infty}^{\infty}g(u)e_{2r+1}(uv,\lambda)\,d\widetilde{\nu}_{\lambda}(u).
\end{align}
Applying \eqref{eq34} for $g\in \mathcal{S}_n(\mathbb{R})$, we get
\begin{equation}\label{eq35}
\int_{-\infty}^{\infty}\delta_{\lambda}^ng(u)e_{2r+1}(uv,\lambda)\,d\widetilde{\nu}_{\lambda}(u)=
(-1)^n|v|^{2n}\int_{-\infty}^{\infty}g(u)e_{2r+1}(uv,\lambda)\,d\widetilde{\nu}_{\lambda}(u).
\end{equation}

\begin{lemma}\label{lem10}
Suppose $g\in \mathcal{S}(\mathbb{R})$ and
\[
a_n(g)(v)=\sum_{k=0}^n\sum_{l=0}^k\frac{1}{(k-l)!}\,\frac{\partial^{2l+1}g(0)}{(2l+1)!}\,v^{2k+1}e^{-v^2},
\]
then $g-a_n(g)\in\mathcal{S}_{n+1}(\mathbb{R})$.
\end{lemma}

\begin{proof} We have $g-a_n(g)\in\mathcal{S}(\mathbb{R})$.
Applying
the  Leibniz rule
for $s=0,1,\dots,n$,
\begin{align*}
\frac{\partial^{2s+1}a_n(g)(0)}{(2s+1)!}&=\frac{1}{(2s+1)!}\sum_{k=0}^s\sum_{l=0}^k\frac{\partial^{2l+1}g(0)}{(k-l)!\,(2l+1)!}
\binom{2s+1}{2k+1}(2k+1)!\partial^{2s-2k}(e^{-u^2})(0)\\
&=\sum_{k=0}^s\sum_{l=0}^k\frac{\partial^{2l+1}g(0)(-1)^{s+k}}{(k-l)!\,(2l+1)!\,(s-k)!}
=
\sum_{l=0}^s\frac{\partial^{2l+1}g(0)}{(2l+1)!}\sum_{m=0}^{s-l}\frac{(-1)^{s+l+m}}{m!\,(s-l-m)!}
\\
&=\frac{1}{(s-l)!}\sum_{l=0}^s\frac{(-1)^{s+l}\partial^{2l+1}g(0)}{(s-l)!\,(2l+1)!}(1-1)^{s-l}=\frac{\partial^{2s+1}g(0)}{(2s+1)!}.
\end{align*}
\end{proof}

\begin{proposition}\label{lem11} Let
$\lambda>-1/2$ and $r\in\mathbb{Z}_+$.
We have  $\mathcal{F}_{r}^{\lambda}(\mathcal{S}_{r}(\mathbb{R}))=\mathcal{S}_{r}(\mathbb{R})$.
\end{proposition}
If $r=0$ we recover the result by de Jeu \cite{jeu} for the Dunkl transform.

\begin{proof} Let $r\in\mathbb{N}$ and $g\in \mathcal{S}_{r}(\mathbb{R})$. It is enough to show that  $g$ satisfies the condition:
\begin{equation}\label{eq36}
\partial^{m}(v^{2n}\mathcal{F}_{r}^{\lambda}(g)(v))\ \text{is bounded for any}\ m\in\mathbb{Z}_+,\ n\ge r+1.
\end{equation}
We have
\[
\partial^{m}(v^{2n}\mathcal{F}_{r}^{\lambda}(g)(v))=\partial^{m}(v^{2n}\mathcal{F}_{r}^{\lambda}(g-a_{n-1}(g))(v))
+\partial^{m}(v^{2n}\mathcal{F}_{r}^{\lambda}(a_{n-1}(g))(v)).
\]
Since  $\partial^{2l+1}g(0)=0$, $l=0,1,\dots,r-1$, then
\[
a_{n-1}(g)(v)=\sum_{k=r}^{n-1}\sum_{l=r}^k\frac{1}{(k-l)!}\,
\frac{\partial^{2l+1}g(0)}{(2l+1)!}\,v^{2k+1}e^{-v^2}\in\mathcal{S}_{r}(\mathbb{R}).
\]
By  \eqref{eq31}, condition  \eqref{eq36} is valid for $a_{n-1}(g)$. By Lemma \ref{lem10}, $g-a_{n-1}(g)\in \mathcal{S}_n(\mathbb{R})$.
In light of~\eqref{eq35}, $v^{2n}\mathcal{F}_{r}^{\lambda}(g-a_{n-1}(g))(v)$ is the  $\mathcal{F}_{r}^{\lambda}$-transform of the function
$(-1)^n\delta_{\lambda}^n(g-a_{n-1}(g))\in \mathcal{S}(\mathbb{R})$. Applying for this transform inequality \eqref{eq29}, we get the property \eqref{eq36} for  $g-a_{n-1}(g)$. Therefore, $\mathcal{F}_{r}^{\lambda}(g)\in\mathcal{S}(\mathbb{R})$. By virtue of \eqref{eq28} and \eqref{eq30}, $\mathcal{F}_{r}^{\lambda}(g)\in\mathcal{S}_{r}(\mathbb{R})$.
\end{proof}

\begin{remark}
An alternative proof of Proposition  \ref{lem11} reads as follows.
Let $g\in\mathcal{S}_{r}(\mathbb{R})$. In light of \eqref{eq24}--\eqref{eq26}, we have the following representation
\begin{align*}
\mathcal{F}_{r}^{\lambda}(g)(v)&=\int_{-\infty}^{\infty}j_{\lambda}(uv)g(u)\,d\widetilde{\nu}_{\lambda}(u)+
\frac{i(-1)^{r+1}v^{2r+1}}{2^{2r+1}(\lambda+1)_{2r+1}}\int_{-\infty}^{\infty}u^{2r+1}j_{\lambda+2r+1}(uv)g(u)\,d\widetilde{\nu}_{\lambda}(u)\\
&=\int_{0}^{\infty}j_{\lambda}(uv)g_{e}(u)\,d\nu_{\lambda}(u)+i(-1)^{r+1}v^{2r+1}\int_{0}^{\infty}j_{\lambda+2r+1}(uv)u^{-(2r+1)}g_{o}(u)
\,d\nu_{\lambda+2r+1}(u)\\
&=H_{\lambda}(g_e)(v)++i(-1)^{r+1}v^{2r+1}H_{\lambda+2r+1}(u^{-(2r+1)}g_o)(v)
=F_1(v)+v^{2r+1}F_2(v),
\end{align*}
where even functions $F_1,F_2\in\mathcal{S}(\mathbb{R})$. Therefore, $\mathcal{F}_{r}^{\lambda}(g)\in\mathcal{S}_{r}(\mathbb{R})$. 
The first  proof was given to underline the important properties of $\mathcal{F}_{r}^{\lambda}$-transform and its kernel.
\end{remark}

Since $f\in\mathcal{S}(\mathbb{R})$ implies $Af\in \mathcal{S}_r(\mathbb{R})$, Proposition \ref{lem11} yields the following result.

\begin{corollary}
Let $a=\frac{2}{2r+1}$ and $\kappa\ge\frac{r}{2r+1}$. 
If $f\in\mathcal{S}(\mathbb{R})$, then $\mathcal{F}_{\kappa,a}(f)((2r+1)^{-(r+1/2)} v^{2r+1})\in\mathcal{S}_r(\mathbb{R})$ or, equivalently,
$\mathcal{F}_{\kappa,a}(f)(v)=g((2r+1)^{1/2}\,v^{1/(2r+1)}),$ $g\in\mathcal{S}_r(\mathbb{R})$; cf. \eqref{eq27}.
\end{corollary}

Now we discuss the case $\lambda=-1/2$.
\begin{remark}
If $\lambda=-1/2$, $r\in\mathbb{Z}_+$, then
\[
\delta_{-1/2}g(u)=\partial^{2}g(u)-r(2r+1)\,\frac{g(u)-g(-u)}{u^2},\quad e_{1}(uv,-1/2)=e^{-iuv} \ (r=0).
\]
Taking into account \eqref{eq7} and  passing to the limit in \eqref{eq7} as $\lambda\to -1/2$, we deduce that for $r\ge 1$
\begin{align*}
e_{2r+1}(uv,-1/2)&=\cos{}(uv)+
i(-1)^{r+1}\,\frac{(uv)^{2r+1}}{2^{2r+1}(1/2)_{2r+1}}
j_{2r+1/2}(uv)\\
&=e^{-iuv}-(r+1/2)\int_{-1}^{1}\sum_{s=0}^{r-1}(-1)^s\binom{r}{s+1}\frac{(r+3/2)_s}{s!}\,(1-t^2)^ste^{-iuvt}\,dt.
\end{align*}
Taking this into account  and analyzing the proofs above,
we note that
all mentioned results in this section for the transform $\mathcal{F}_{r}^{\lambda}$ in the case $\lambda>-1/2$ are also valid for $\lambda=-1/2$. In particular,
  $\mathcal{F}_{r}^{-1/2}$, $r\in\mathbb{Z}_+$, are the unitary transforms in the non-weighted  $L^2(\mathbb{R},dx)$, where
 $\mathcal{F}_{0}^{-1/2}$ corresponds to the classical Fourier transform.
\end{remark}

\end{document}